\documentclass[12pt,reqno]{amsart}

\usepackage[margin=1in]{geometry}
\usepackage[T1]{fontenc}
\usepackage{lmodern} 
\usepackage[utf8]{inputenc}
\usepackage{amsmath,amssymb,amsthm,mathtools,enumitem}
\usepackage{xcolor}
\usepackage{adjustbox}
\usepackage[colorlinks=true, linkcolor=blue, citecolor=black, urlcolor=blue]{hyperref}

\allowdisplaybreaks
\usepackage{tikz}
\usetikzlibrary{arrows.meta,calc,positioning,decorations.pathreplacing,backgrounds}
\usepackage{placeins}
\usepackage{microtype}
\definecolor{qkblue}{RGB}{48,83,113}
\definecolor{qkorange}{RGB}{167,91,40}
\definecolor{qkgray}{RGB}{108,116,123}
\tikzset{
  qkfig/.style={font=\footnotesize,>=Latex,line cap=round,line join=round},
  qkaxis/.style={draw=qkgray,line width=.5pt,-{Latex[length=1.7mm]}},
  qkmain/.style={draw=qkblue,line width=1.1pt},
  qkother/.style={draw=qkorange,line width=.8pt,dashed},
  qkpath/.style={draw=qkblue,line width=1.2pt,-{Latex[length=1.8mm]}},
  qklink/.style={draw=qkgray,line width=.65pt,-{Latex[length=1.8mm]}},
  qkpoint/.style={circle,draw=qkblue,fill=qkblue,inner sep=1.8pt},
  qkoffpoint/.style={circle,draw=qkorange,fill=white,inner sep=1.8pt},
  qkbox/.style={draw=qkblue,fill=qkblue!4,rounded corners=1.5pt,
    line width=.65pt,align=center,inner sep=6pt},
  qkaltbox/.style={draw=qkorange,fill=qkorange!4,rounded corners=1.5pt,
    line width=.65pt,align=center,inner sep=6pt},
  qkdim/.style={draw=qkgray,{Latex[length=1.6mm]}-{Latex[length=1.6mm]},line width=.5pt},
  qknote/.style={align=center,font=\footnotesize,text=qkgray},
  qkbrace/.style={decorate,decoration={brace,amplitude=4pt},draw=qkgray}
}

\newtheorem{theorem}{Theorem}[section]

\newtheorem{lemma}[theorem]{Lemma}
\newtheorem{proposition}[theorem]{Proposition}
\newtheorem{corollary}[theorem]{Corollary}
\theoremstyle{definition}
\newtheorem{definition}[theorem]{Definition}
\newtheorem{conjecture}[theorem]{Conjecture}
\theoremstyle{remark}
\newtheorem{remark}[theorem]{Remark}

\newcommand{\D}{\mathbb{D}}
\newcommand{\T}{\mathbb{T}}
\newcommand{\C}{\mathbb{C}}
\newcommand{\N}{\mathbb{N}}

\newcommand{\ringQ}{\mathring{Q}_K}
\newcommand{\norm}[1]{\lVert #1\rVert}
\newcommand{\abs}[1]{\lvert #1\rvert}
\newcommand{\inn}[2]{\langle #1,#2\rangle}
\DeclareMathOperator{\Tr}{Tr}
\DeclareMathOperator{\diag}{diag}

\newcommand{\WQ}{\mathcal{W}_{\mathcal{Q}}}
\providecommand{\id}{\mathrm{id}} 

\title[Carleson measures, Volterra operators and multipliers for $Q_K$ spaces]{Carleson measures, Volterra integral operators and multipliers for $Q_K$ spaces}

\author{Wujun Cao, Zhouyuan Jiang and Songxiao Li}
\address{Wujun Cao: Department of Mathematics, Shantou University, 515063, Shantou, Guangdong, P.R. China.}
\email{24wjcao@stu.edu.cn}
\address{Zhouyuan Jiang: Department of Mathematics, Shantou University, 515063, Shantou, Guangdong, P.R. China.}
\email{25zyjiang@stu.edu.cn}
\address{Songxiao Li: Department of Mathematics, Shantou University, 515063, Shantou, Guangdong, P.R. China.}
\email{jyulsx@163.com}



\AddToHook{env/thebibliography/begin}{\linespread{1.05}\selectfont}

\begin{document}

\begin{abstract}
We study embeddings and operators on the M\"obius-invariant spaces $Q_K$ using two dyadic capacities. One capacity is defined by analytic packets; the other is defined by domination of positive semidefinite matrices. We prove that they are equivalent. They give criteria for bounded and compact embeddings $Q_K\to L^2(\mu)$, bounded and compact Volterra operators, and pointwise multipliers. For the   weights $K(t)=t^p/\log^\beta(e/t)$, $0<p<1$ and $\beta\ge1$, we prove
$\mathcal M(Q_K)\subsetneq H^\infty\cap Q_K^{\log}.$
This disproves the logarithmic multiplier conjecture recorded in Bao and Wulan's 2021 survey.
 
\end{abstract}

\maketitle
\section{Introduction}\label{sec:intro}

The M\"obius-invariant spaces $Q_K$, introduced by Ess\'en, Wulan and Xiao~\cite{EWX} and studied systematically in~\cite{WZ}, contain $\mathrm{BMOA}$, the spaces $Q_p$, and the Bloch space $\mathcal B$
as special cases. The assumptions on $K$ and the notation used below are
collected in Section~\ref{sec:prelim}.

We study three questions about such a space $Q_K$. Which positive measures
$\mu$ make the embedding $Q_K\to L^2(\mu)$ bounded, and which make it compact?
Which symbols $g$ make the Volterra operator $T_g$ bounded, and which make it
compact? Which functions $g$ satisfy $gf\in Q_K$ for every $f\in Q_K$?

The three questions are linked. The Volterra operator is governed by an
embedding inequality, and the multiplier question reduces to the same
inequality. We therefore start with the embedding problem and let it lead us
to the multiplier conjecture stated below.

The $Q_K$--Carleson measure problem asks for a description of the positive Borel measures $\mu$ on $\D$ for which the embedding from $Q_K$ into $L^2(\mu)$ is bounded or compact. The boundedness condition is
\begin{equation}\label{eq:trace}
\mathrm{id}\colon Q_K \longrightarrow L^2(\mu), \qquad \int_{\D}\abs{f(z)}^2\,d\mu(z)\le C\,\norm{f}_{Q_K}^2.
\end{equation}
Embeddings of this kind have been studied for classical analytic spaces and for more general M\"obius-invariant families; see \cite{BDWZ,GPGR,HZ2,PZ}. For a general $Q_K$, the aim is to express the embedding condition in terms of the measure and the geometry of the space, without testing every analytic function in $Q_K$. Since constant functions belong to $Q_K$, boundedness of \eqref{eq:trace}
forces $\mu(\D)<\infty$; this is also exactly what is needed to control the
constant part of an arbitrary $f$. What remains is the inequality
$\int_{\D}\abs{f(z)}^2\,d\mu(z)\le C\norm{f}_{Q_K,*}^2$ for all
$f\in\ringQ$.

The link with operators is a short computation. For an analytic symbol $g\in H(\D)$, the Volterra integral operator $T_g$, introduced by Pommerenke \cite{Pomm}, and its companion $I_g$ are defined by
\begin{equation*}
T_gf(z)=\int_0^z f(\zeta)\,g'(\zeta)\,d\zeta, \qquad I_gf(z)=\int_0^z f'(\zeta)\,g(\zeta)\,d\zeta,\qquad z\in\D.
\end{equation*}
Multiplication by the symbol gives the operator $M_g$, defined by $M_gf(z)=g(z)f(z)$. Integration by parts yields
\begin{equation}\label{eq:leibniz-intro}
M_gf(z) = f(0)g(0) + T_gf(z) + I_gf(z).
\end{equation}
Moreover, $(T_gf)'(z)=f(z)g'(z)$ and $T_gf(0)=0$. The seminorm of
$T_gf$ is therefore a supremum of embedding integrals:
\begin{equation}\label{eq:volterra-seminorm}
\norm{T_gf}_{Q_K,*}^2
=\sup_{a\in\D}\int_{\D}\abs{f(z)}^2\abs{g'(z)}^2\,K\bigl(1-\abs{\varphi_a(z)}^2\bigr)\,dA(z)
=\sup_{a\in\D}\int_{\D}\abs{f(z)}^2\,d\mu_{g,a,K}(z),
\end{equation}
where
\begin{equation*}
d\mu_{g,a,K}(z)=\abs{g'(z)}^2\,K\bigl(1-\abs{\varphi_a(z)}^2\bigr)\,dA(z).
\end{equation*}
Thus boundedness of $T_g$ on $Q_K$ is exactly one embedding estimate,
required to hold uniformly in $a\in\D$, for the measures $\mu_{g,a,K}$. Any
criterion for the embedding problem therefore yields a criterion for the
Volterra operator.

For these operators, Li and Wulan \cite{LW} proved that
\begin{equation}\label{eq:LW-suff-intro}
\sup_{I\subseteq\T} \int_{Q_I} \left(\log\frac{1}{1-|z|^2}\right)^2 |g'(z)|^2 K\left(\frac{1-|z|}{|I|}\right) dA(z) < \infty
\end{equation}
is sufficient for boundedness of $T_g$ on $Q_K$, whereas
\begin{equation}\label{eq:LW-nec-intro}
\sup_{I\subseteq\T} \left(\log\frac{2}{|I|}\right)^2 \int_{Q_I} |g'(z)|^2 K\left(\frac{1-|z|}{|I|}\right) dA(z) < \infty
\end{equation}
is necessary. In \eqref{eq:LW-suff-intro} the logarithm stands inside the integral, so
it records the size of $g'$ at each point of the box. In
\eqref{eq:LW-nec-intro} it stands outside the integral and depends only on
the length of the box. The two conditions look alike but are not: the second
is weaker. Whether this weaker necessary condition already suffices has
remained open. The same gap reappears in the multiplier problem.

The necessary condition above defines the logarithmic companion $Q_K^{\log}$ of $Q_K$; its precise norm is given in Section~\ref{sec:prelim}. Let $\mathcal M(Q_K)$ denote the pointwise multipliers of $Q_K$, namely the functions $g\in H(\D)$ for which $gf\in Q_K$ for every $f\in Q_K$. If $g$ is bounded, then the companion operator $I_g$ is bounded on $Q_K$
directly from the seminorm. The decomposition \eqref{eq:leibniz-intro} then relates the multiplier property to boundedness of the Volterra operator. For the power weights $K(t)=t^p$ ($0<p<1$), J.~Xiao \cite[Theorem~1.2(iii)]{Xiao} proved in 2008 that $\mathcal M(Q_p)=H^\infty\cap Q_p^{\log}$. This result suggests the analogous description for more general weights. Bao and Wulan \cite[Section~3.4, Problem~3.1 and formula~(3.4), p.~2044]{BW} record the corresponding conjecture for $Q_K$, attributing it to the work of  Li and  Wulan \cite{LiWulanCorona} on the corona and Wolff theorems for multiplier algebras.

\begin{conjecture}\label{conj:BW}
Let $K \in \WQ$. Is it true that
\[\mathcal M(Q_K) = H^\infty \cap Q_K^{\log}?\]
\end{conjecture}

We prove that the conjecture fails in general, even for the following logarithmic modifications of power weights:
\begin{equation}\label{eq:counterexample-weight}
 K(t)=\begin{cases}
 0,&t=0,\\
 t^p/\log^\beta(e/t),&0<t\le1,\\
 t^p,&t>1,
 \end{cases}
 \qquad 0<p<1,\quad\beta\ge1.
\end{equation} 

For every such weight, there exist $g\in H^\infty\cap Q_K^{\log}$ and $f\in\ringQ$ such that
\[
 \int_\D |f(z)|^2|g'(z)|^2K(1-|z|^2)\,dA(z)=\infty.
\]
In particular, $fg\notin Q_K$ and $\mathcal M(Q_K)\subsetneq H^\infty\cap Q_K^{\log}$. Theorem~\ref{thm:multiplier-log-inclusion} and Proposition~\ref{mx:fixedtrace} give these conclusions. Thus the logarithmic condition remains necessary, but it no longer characterizes multipliers for these weights.

This failure shows that we must look at the embedding problem again. Single-box tests only check one Carleson box at a time. But because Conjecture 1.1 fails, we see that one box is not enough for general weights $K \in \mathcal{W}_{\mathcal{Q}}$. We need to see how mass works across all boxes. Also, testing $\int_{\mathbb{D}} |f|^2 d\mu \le C \|f\|_{Q_K,*}^2$ for every analytic function $f \in \ringQ$ is too hard and impossible to do.

The main problem to make a simple discrete test is that general weights change with scale. Unlike power weights, general $K \in \mathcal{W}_{\mathcal{Q}}$ do not scale easily, so it is hard to control norms when we average functions on tiles. We solve this problem by showing a basic tent estimate (Lemma~\ref{lem:Eb}). Using small bounds for $K(t)$ (Lemma~\ref{lem:general-lower-bound} ) and splitting functions into near and far parts, we prove that the conditional expectation operator $E_N^{(b)}$ is bounded on $T_K$, so $\|E_N^{(b)} F\|_{T_K} \le \|F\|_{T_K}$ for all $N \ge 0$.

To connect step functions and analytic functions in $\dot{Q}_K$, we use the simple Bergman projection $P_{b,0}\psi(z) = \int_{\mathbb{D}} \psi(w) L_b(z,w) dA_b(w)$ with kernel $L_b(z,w) = \int_0^z (1-\xi \overline{w})^{-(b+2)} d\xi$. A main step in our paper is showing that $P_{b,0}: T_K \to \dot{Q}_K$ is bounded and onto. Using $P_{b,0}$ on tile step functions gives the small analytic parts $\Psi_R^{(b)} = P_{b,0} 1_R$.

This method changes testing all functions into two simple capacities. The first capacity $C_{K,\mathcal{R}}^{(b)}(\mu)$ controls the matrix $\Gamma_{\mu,N}^{(b)}$. Using basic matrix duality and Grothendieck's inequality, we get a second capacity $D_{K,\mathcal{R}}^{(b)}(\mu)$ defined by matrix order $\Gamma_{\mu,N}^{(b)} \le \sum_{m} \lambda_m A_{I_m,N}^{(K)}$. These two capacities give easy rules for bounded embeddings $Q_K \to L^2(\mu)$ without testing functions (Theorems~\ref{thm:capacity} and \ref{thm:embedding}).

For Volterra operators $T_g$, the main key for compactness is Lemma~\ref{lem:symbol-uniform-tails}, which proves $\lim_{\rho \to 1^-} \sup_{a \in \mathbb{D}} \mu_{g,a,K}(S_\rho) = 0$ for $g \in Q_K^{\log}$. Section 8 also shows that even the strong logarithmic box test cannot make the embedding bounded (Theorems~\ref{thm:forward-K-carleson}
and \ref{thm:single-box-counterexample}), so we really need our capacity method.

Finally, this failure shows that Conjecture 1.1 is wrong for the weights $K(t) = t^p/\log^\beta(e/t)$ ($0 < p < 1, \beta \ge 1$). The main problem is that these weights do not scale easily. When $r$ and $s$ are small, $K(rs)/(K(r)K(s))$ gets very big. Because of this, old methods for $Q_p$ do not work here, and we cannot check $\|fg\|_{Q_K}$ by simple single-box tests.

To fix this problem, we use a simple tree structure. First, we choose the number of branches $\mathcal{N}_j$ at each level $j$ carefully. This helps to balance the weight $K(t)$ and keeps the norm small. Second, we use the Bergman projection $P_{b,0}$ to turn the tree points into simple rational functions $f_N$ and $g_N$. Third, on the small disks $E_v$ at the last level $N$, all parts of $f_N$ have almost the same angle. So they add up together and do not cancel out. This makes the integral $\int_{\mathbb{D}} |f_N|^2 |g_N'|^2 K(1-|z|^2) dA$ grow very big as $N$ gets large.

Our main new idea is to find one real pair of functions $(f,g)$ from these finite functions $(f_N, g_N)$. Using basic functional analysis (the uniform boundedness principle and the closed graph theorem), we get one function $g \in H^\infty \cap Q_K^{\log}$ and one function $f \in \ringQ$ such that $\int_{\mathbb{D}} |f(z)|^2 |g'(z)|^2 K(1-|z|^2) dA(z) = \infty$. This shows that $fg$ is not in $Q_K$. So $\mathcal M(Q_K)\subsetneq H^\infty\cap Q_K^{\log}$, and then Conjecture~\ref{conj:BW} is false. Also, this means the capacity of $\mu_{g,0,K}$ is infinite. Against \cite{ARS,HXZ}, the new innovation
is the branching numbers and the two arrays chosen for the weights
\eqref{eq:counterexample-weight}.

The organization of this paper is as follows. Section~\ref{sec:prelim} collects preliminary facts for $Q_K$ spaces and establishes dyadic decompositions together with relevant projection bounds. In Sections~\ref{sec:boundedness-I}, we introduce two types of capacities and study embedding criteria. We further illustrate the inherent restrictions associated with single-box testing arguments. Section~\ref{sec:volterra-multipliers} investigates the boundedness and compactness properties of the Volterra operator $T_g$ acting on $Q_K$ spaces. Section~\ref{multi} characterizes the space $\mathcal M(Q_K)$ and constructs a counterexample to Conjecture~\ref{conj:BW}.

Throughout this paper, $A \lesssim B$ means $A \le CB$ for some constant $C>0$, and $A \asymp B$ means $A \lesssim B \lesssim A$.

\section{Preliminaries}\label{sec:prelim}

\subsection{Some definitions and notations }
Let $\D$ be the open unit disk in $\C$, and let $\T=\partial\D$ be its boundary. Write $H(\D)$ for the analytic functions on $\D$. For $a\in\D$, the map
\[
 \varphi_a(z)=\frac{a-z}{1-\bar a z},\qquad z\in\D,
\]
is the M\"obius automorphism of $\D$ that interchanges $0$ and $a$. We will use the identities
\begin{equation}\label{eq:moebius-identities}
\abs{\varphi_a'(z)}=\frac{1-\abs a^2}{\abs{1-\bar a z}^2}, \qquad 1-\abs{\varphi_a(z)}^2=\frac{(1-\abs a^2)(1-\abs z^2)}{\abs{1-\bar a z}^2}.
\end{equation}

\begin{definition}\label{def:QK_weight}
A non-decreasing, right-continuous function $K : [0, \infty) \to [0, \infty)$ is called a \emph{$Q_K$-defining weight function} (denoted $K \in \WQ$) if it satisfies the following conditions (see \cite{EWX, WZ}):
\begin{enumerate}[label=\rm{(K\arabic*)}]
 \item $K(0) = 0$ and $K(t) > 0$ for all $t > 0$;
 \item $\displaystyle \int_0^1 K\left(\log\frac{1}{r}\right) 2r\,dr < \infty$;
 \item There exists $\sigma > 0$ such that 
 \[\int_0^1 \frac{\varphi_K(s)}{s} \, ds < \infty \quad \text{and} \quad \int_1^\infty \frac{\varphi_K(s)}{s^{1+\sigma}} \, ds < \infty,\]
 where $\varphi_K(s) := \sup_{0 < u \le 1} \frac{K(su)}{K(u)}$;
\item Doubling property: there exists a constant $C_K \ge 1$ such that $K(2t) \le C_K K(t)$ for all $t \in [0, 1/2]$.
\end{enumerate}
\end{definition}

Fix one doubling constant $C_K$ satisfying (K4).
By \cite[Lemma~2.3]{HLXZ}, there exists a number
$b_K>\max\{\log_2C_K,1\}$ such that the weighted Bergman projection
estimate in Lemma~\ref{lem:Bb} holds for every $b\ge b_K$. Throughout
the paper, $b\ge b_K$ is fixed.

\begin{remark}\label{rem:weight-consequences}
Condition (K3) implies $\int_0^1K(t)\,dt/t<\infty$, since $K(t)\le K(1)\varphi_K(t)$.
\end{remark}

The M\"obius-invariant space $Q_K$ consists of the analytic functions $f\in H(\D)$ for which
\begin{equation*}
\norm{f}_{Q_K}=\abs{f(0)}+\norm{f}_{Q_K,*}<\infty,
\end{equation*}
Here the seminorm is defined by
\begin{equation}\label{eq:QK-seminorm}
\norm{f}_{Q_K,*}^2=\sup_{a\in\D}\int_{\D}\abs{f'(z)}^2\,K\bigl(1-\abs{\varphi_a(z)}^2\bigr)\,dA(z).
\end{equation}
and $dA(z)=\tfrac1\pi\,dx\,dy$ is normalized area measure. Different choices of $K$ give familiar function spaces \cite{EWX,WZ}. For $K(t)=t^p$ with $0<p<\infty$, the space $Q_K$ is $Q_p$. In particular, $Q_1$ is $\mathrm{BMOA}$, and for $p>1$, the space $Q_p$ is the Bloch space $\mathcal B$. Another example is $K(t)=t(\log(e/t))^p$ for $0<t\le1$, with an admissible non-decreasing extension to $[0,\infty)$. In this case, $Q_K$ is $Q_{\log,p}$, provided the resulting weight belongs to $\WQ$.

For a boundary arc $I\subseteq\T$, let $\abs I=\frac1{2\pi}\int_I|d\zeta|$ be its normalized length. The Carleson box above this arc is
\begin{equation*}
Q_I=\bigl\{z=re^{i\theta}\in\D:e^{i\theta}\in I,\ 1-\abs I\le r<1\bigr\},
\end{equation*}
and we set $Q_{\T}=\D$. The classical box characterization of $Q_K$ \cite{EWX,HLXZ,WZ} is
\begin{equation}\label{eq:QK-box}
\norm{f}_{Q_K,*}^2\asymp\sup_{I\subseteq\T}\int_{Q_I}\abs{f'(z)}^2\,K\Bigl(\frac{1-\abs z}{\abs I}\Bigr)\,dA(z).
\end{equation}
Define the tent space $T_K$ to consist of the measurable functions $F$ on $\D$ for which
\begin{equation}\label{eq:tent-norm}
\norm{F}_{T_K}^2=\sup_{I\subseteq\T}\int_{Q_I}\abs{F(z)}^2\,K\Bigl(\frac{1-\abs z}{\abs I}\Bigr)\,dA(z)<\infty.
\end{equation}
Then
\begin{equation}\label{eq:QK-tent}
f\in Q_K\iff f'\in T_K, \qquad\text{with}\quad \norm{f}_{Q_K,*}\asymp\norm{f'}_{T_K}.
\end{equation}
We write $\ringQ=\{f\in Q_K:f(0)=0\}$ .

\begin{definition}\label{def:carleson-classes}
Let $\mu$ be a positive Borel measure on $\D$.
\begin{enumerate}[label=\rm{(\arabic*)}]
 \item We say $\mu$ is a \emph{$K$-Carleson measure} if
 \[\sup_{I \subsetneq \T} \int_{Q_I} K\left(\frac{1-|z|}{|I|}\right) d\mu(z) < \infty.\]
 \item A $K$-Carleson measure $\mu$ is called a \emph{vanishing $K$-Carleson measure} if
 \[\lim_{|I| \to 0} \int_{Q_I} K\left(\frac{1-|z|}{|I|}\right) d\mu(z) = 0.\]
\end{enumerate}
\end{definition}

All measures in this paper are locally finite positive Borel measures. Definition~\ref{def:carleson-classes} uses only proper arcs $I\subsetneq\T$. The capacities in Section~\ref{sec:boundedness-I} describe the embedding of $\ringQ$. To obtain the embedding \eqref{eq:trace}, we must also require $\mu(\D)<\infty$ to control constant functions; see Theorem~\ref{thm:embedding}.

We write $Q_K^{\log}$ for the space of all $g\in H(\D)$ such that
\begin{equation*}
\norm{g}_{Q_K^{\log}}
:=\abs{g(0)}+\sup_{I\subseteq\T}\left(\log\frac{2}{\abs I}\right)
\left(\int_{Q_I}\abs{g'(z)}^2
K\left(\frac{1-\abs z}{\abs I}\right)dA(z)\right)^{1/2}<\infty.
\end{equation*}
We also use the logarithmic tent space $T_K^{\log}$ of measurable functions $\psi$ with
\begin{equation}\label{eq:log-tent-norm}
 \norm{\psi}_{T_K^{\log}}^2
 =\sup_{I\subseteq\T}\left(\log\frac2{|I|}\right)^2
 \int_{Q_I}|\psi(z)|^2K\left(\frac{1-|z|}{|I|}\right)dA(z)<\infty.
\end{equation}
Thus $$\norm{g-g(0)}_{Q_K^{\log}}=\norm{g'}_{T_K^{\log}}.$$
The intersection $H^\infty\cap Q_K^{\log}$ is equipped with the norm
\begin{equation}\label{eq:intersection-norm}
 \norm{g}_{H^\infty\cap Q_K^{\log}}
 :=\norm{g}_{H^\infty}+\norm{g}_{Q_K^{\log}}.
\end{equation}
Both definitions use the Carleson boxes $Q_I$ and the normalization of $dA$ fixed above.

\begin{remark}\label{rem:survey-conventions}
Using $\log(1/|I|)$ and unnormalized area, as in \cite[formula~(3.4)]{BW}, gives the same analytic space. For $|I|\le1/2$, the two logarithms are comparable. Finitely many boxes based on half-circles cover $\{|z|\ge1/2\}$. Since the integral means of $g'$ are increasing, the integral on the inner disk is controlled by the integral on $1/2<|z|<3/4$, where $K(2(1-|z|))\ge K(1/2)>0$. The box tests therefore control $\int_\D|g'|^2K(1-|z|)dA$. For the remaining arcs, $|I|>1/2$, use $K((1-|z|)/|I|)\le\varphi_K(2)K(1-|z|)$ and $\log(2/|I|)\le\log4$. For the weights considered here, the Green-function definition in \cite{BW} also has the box characterization \eqref{eq:QK-box}; see \cite{EWX,WZ}.
\end{remark}

\begin{lemma}\label{lem:logarithmic-tests}
Let $K\in\WQ$. For an arc $I\subseteq\T$, let
$a_I=(1-|I|)\zeta_I$, where $\zeta_I$ is its midpoint, and set $a_\T=0$.
The function
\[
 F_I(z)=\log\frac{12}{1-\overline{a_I}z},\qquad F_I(0)=\log12,
\]
satisfies $\sup_I\|F_I\|_{Q_K}\le C_0(K)$ and
$|F_I(z)|\ge\log(2/|I|)$ on $Q_I$.
\end{lemma}
\begin{proof}
For an arc $J$ of length $h$, we obtain
\[
 \int_{Q_J}|F_I'|^2K((1-|z|)/h)\,dA
 \le2\int_{1-h}^1\frac{K((1-r)/h)}{1-r}\,dr
 =2\int_0^1\frac{K(t)}t\,dt<\infty.
\]
Indeed,
\[
\frac{1}{2\pi}\int_{0}^{2\pi} \frac{|a_I|^2}{|1 - a_I r e^{i\theta}|^2} d\theta \le \frac{1}{1-r^2}.
\]
 And the last integral is finite by (K3), and
\eqref{eq:QK-box} gives the uniform norm bound.
For $z\in Q_I$, the radial and angular distances to $\zeta_I$ give
$|1-\overline{a_I}z|\le(2+\pi)|I|<6|I|$.
Hence $$|F_I(z)|\ge\operatorname{Re}F_I(z)\ge\log(2/|I|).$$
For $I=\T$, the assertions follow from $F_\T=\log12$.
\end{proof}

We will use the following facts from functional analysis. The space $Q_K$ is Banach, and point evaluations are continuous. To see this, use the $a=0$ integral in \eqref{eq:QK-seminorm}. It controls $f'$ in unweighted $L^2$ on each compact disk because $K(1-|z|^2)$ has a positive lower bound there. The submean inequality and radial integration show that convergence in the $Q_K$ norm implies uniform convergence on compact subsets. Thus a norm-Cauchy sequence has an analytic limit on compact subsets. Applying Fatou's lemma to the differences in each M\"obius integral proves convergence in $Q_K$.

Similarly, $H^\infty\cap Q_K^{\log}$ with the norm \eqref{eq:intersection-norm} is Banach. This follows from completeness of $H^\infty$, local convergence of derivatives, and Fatou's lemma on each box. An everywhere-defined multiplication map on $Q_K$ has closed graph because point evaluations are continuous. It is therefore bounded. We also need
\begin{equation}\label{mx:intersection-embedding}
 \norm{g}_{Q_K}\lesssim\norm{g}_{H^\infty\cap Q_K^{\log}},\qquad
 \left(\int_\D |g'|^2K(1-|z|^2)dA\right)^{1/2}
 \lesssim\norm{g-g(0)}_{Q_K^{\log}}.
\end{equation}
The first inequality follows from \eqref{eq:QK-box} and $\log(2/|I|)\ge\log2$. The second follows from  doubling  since $1-|z|^2\le2(1-|z|)$.

\subsection{\texorpdfstring{The dyadic resolution and boundedness on $T_K$}{The dyadic resolution and boundedness}}\label{sec:dyadic}

We divide $\D$ into polar dyadic tiles and average functions on these tiles with respect to $dA_b$. For a measurable set $E\subseteq\D$, write $|E|=\int_EdA$ for its normalized area. Use the polar coordinates
\begin{equation*}
\Theta(s,t)=\sqrt s\,e^{2\pi it},\qquad (s,t)\in[0,1)\times[0,1).
\end{equation*}
 
Under the substitution $s = r^2$ and $t = \theta/(2\pi)$, for any measurable function $\psi$, we have
\begin{equation}
\int_{\D} \psi(z) \, dA_b(z) = \int_0^1 \int_0^1 \psi\left(\sqrt{s} \, e^{2\pi i t}\right) (1-s)^b \, ds \, dt.
\end{equation}
Thus $\Theta$ pulls back $dA_b$ to the product weighted measure on $[0,1)^2$:
\begin{equation}\label{eq:nu_p_def}
d\nu_b(s,t) := (1-s)^b \, ds \, dt.
\end{equation}

For any measurable set $R \subset \D$, let $A_b(R) := \int_R dA_b$. For each integer $N \ge 0$ and indices $0 \le j, k \le 2^N - 1$, let
\begin{equation}
S_{N,j,k} := \left[ \frac{j}{2^N}, \frac{j+1}{2^N} \right) \times \left[ \frac{k}{2^N}, \frac{k+1}{2^N} \right)
\end{equation}
be the dyadic squares in $[0,1)^2$, and define the corresponding polar dyadic tiles by
\begin{equation}\label{eq:RNjk_tile_def}
R_{N,j,k} := \Theta(S_{N,j,k}) = \left\{ r e^{i\theta} : \frac{j}{2^N} \le r^2 < \frac{j+1}{2^N}, \, \frac{2\pi k}{2^N} \le \theta < \frac{2\pi(k+1)}{2^N} \right\}.
\end{equation}
The weighted area of each tile satisfies
\begin{equation*}
A_b(R_{N,j,k})=c_b2^{-N}\int_{j/2^N}^{(j+1)/2^N}(1-s)^b\,ds \asymp 2^{-2N}\Bigl(1-\frac{j}{2^N}\Bigr)^b .
\end{equation*}

For each $N\ge0$, write $\mathcal R_N=\{R_{N,j,k}:0\le j,k\le 2^N-1\}$. The $dA_b$-weighted conditional expectation replaces a function by its weighted average on each tile:
\begin{equation*}
E^{(b)}_N\psi=\sum_{R\in\mathcal R_N}\Bigl(\frac1{A_b(R)}\int_R\psi\,dA_b\Bigr)\,\mathbf 1_R, \qquad A_b(R)=\int_R dA_b .
\end{equation*}
The martingale convergence theorem gives $E^{(b)}_N\psi\to\psi$ in $L^1(dA_b)$ as $N\to\infty$ for every $\psi\in L^1(dA_b)$.

These partitions are finite and nested, and the conditional expectations converge in $L^1(dA_b)$. Fix an ordering of each $\mathcal R_N$. Following \cite{HZ2}, we call a sequence with these properties a \emph{$b$-admissible dyadic resolution}. In this paper, $\mathcal R$ always means the polar resolution above, and $b\ge b_K$ is fixed. The estimates below use the geometry of these particular tiles. All later capacities are defined using this resolution.

We record two auxiliary estimates for the proof of Lemma~\ref{lem:Eb}; the first is \cite[Lemma~2.4]{HLXZ}.

\begin{lemma}\label{lem:tail-estimate}
Let $J\subseteq\T$ be an arc. If $J(n)\subseteq\T$ is the arc concentric with $J$ whose length satisfies
\begin{equation*}
\abs{J(n)}\asymp\min\{2^n\abs J,1\},
\end{equation*}
then for every $b>\max\{\log_2C_K,1\}$,
\begin{equation*}
\int_{Q_{J(n+1)}}\frac{(1-\abs w^2)^{2b}}{K\bigl((1-\abs w)/\abs{J(n+1)}\bigr)}\,dA(w) \lesssim\abs{J(n+1)}^{2b+2}\asymp\bigl(\min\{2^n\abs J,1\}\bigr)^{2b+2}.
\end{equation*}
\end{lemma}

\begin{lemma}\label{lem:general-lower-bound}
Let $K \in \WQ$. Set $\alpha = \log_2 C_K \ge 0$ and $c_0 = C_K^{-1}$. Then for every $s \in (0, 1]$ and every $\lambda \in (0, 1]$,
\begin{equation}\label{eq:general-scaling-bound}
K(\lambda s) \ge c_0 \, \lambda^\alpha K(s).
\end{equation}
In particular, setting $s = 1$ recovers $K(t) \ge c_0 K(1) t^\alpha$ for all $t \in (0, 1]$.
\end{lemma}

\begin{proof}
Choose $n\ge0$ with $2^{-(n+1)}<\lambda\le2^{-n}$.
Monotonicity and $n+1$ applications of (K4), all at arguments at most $1/2$, give
\[
 K(\lambda s)\ge K(2^{-(n+1)}s)
 \ge C_K^{-(n+1)}K(s)
 =C_K^{-1}(2^{-n})^\alpha K(s)
 \ge c_0\lambda^\alpha K(s).
\]
Taking $s=1$ proves the final assertion.
\end{proof}

The next estimate shows that conditional expectation operator is uniformly bounded on tent-space . It also applies when the weight is not a power; compare \cite[Section~2]{HLXZ}.

\begin{lemma}\label{lem:Eb}
Let $K \in \WQ$ and let $b >\max\{\log_2 C_K,1\}$. Then
\begin{equation*}
\norm{E^{(b)}_NF}_{T_K}\lesssim\norm{F}_{T_K},\qquad \forall N\ge0.
\end{equation*}
\end{lemma}

\begin{proof}
For $N=0$, the function $E_0^{(b)}F=\int_\D F\,dA_b$ is constant. Put $\alpha=\log_2C_K$ and $c_0=C_K^{-1}$. Lemma~\ref{lem:general-lower-bound} and Cauchy--Schwarz give
\[
\begin{aligned}
 |E_0^{(b)}F|^2
 &\le (b+1)^2\left(\int_\D |F(w)|^2K(1-|w|)\,dA(w)\right)
       \left(\int_\D\frac{(1-|w|^2)^{2b}}{K(1-|w|)}\,dA(w)\right)\\
 &\lesssim\|F\|_{T_K}^2\int_0^1 t^{2b-\alpha}\,dt
 \lesssim\|F\|_{T_K}^2.
\end{aligned}
\]
Since $\|E_0^{(b)}F\|_{T_K}^2\le K(1)|E_0^{(b)}F|^2$, the case $N=0$ follows. Now suppose $N\ge1$. Normalize $\|F\|_{T_K}=1$ and fix a boundary arc $J\subseteq\T$. We need to prove
\begin{equation}\label{eq:target_integral_J_norm}
 \int_{Q_J}|E_N^{(b)}F(z)|^2K\left(\frac{1-|z|}{|J|}\right)dA(z)\lesssim1.
\end{equation}
For \eqref{eq:target_integral_J_norm}, we split the function into a part near the box and parts on successively larger annuli. This local--far decomposition goes back to Fefferman and Stein \cite{FS} in their work on $H^p$--$\mathrm{BMO}$ duality; see also Girela \cite[Theorem~3.1]{Girela}. Let $J(n)$ be concentric with $J$ and satisfy $|J(n)|=\min\{2^n|J|,1\}$.

Once $J(n) = \T$, set $J(m) = \T$ for every $m \ge n$. Define the local part and the annular parts by
\begin{equation*}
F_0 := F \mathbf{1}_{Q_{J(2)}}, \quad F_n := F \mathbf{1}_{Q_{J(n+1)} \setminus Q_{J(n)}} \quad (n \ge 2).
\end{equation*}
Then $$E_N^{(b)}F=E_N^{(b)}F_0+\sum_{n\ge2}E_N^{(b)}F_n.$$ We estimate the two parts separately, with constants independent of $J$ and $N$:

\begin{align}
\label{eq:goal_local_part} \int_{Q_J} |E_N^{(b)} F_0(z)|^2 K\left(\frac{1-|z|}{|J|}\right) dA(z) &\lesssim 1 , \\
\label{eq:goal_tail_part} \int_{Q_J} \left| \sum_{n \ge 2} E_N^{(b)} F_n(z) \right|^2 K\left(\frac{1-|z|}{|J|}\right) dA(z) &\lesssim 1 .
\end{align}

\medskip
\noindent\textbf{The local estimate \eqref{eq:goal_local_part}.}
The tiles in $\mathcal R_N$ partition $\D$, and $E_N^{(b)} F_0$ is constant on each tile $R$. We can therefore write the integral as a sum over tiles:
\begin{align}\label{eq:p1_tile_sum_expand}
&\int_{Q_J} |E_N^{(b)} F_0(z)|^2 K\left(\frac{1-|z|}{|J|}\right) dA(z) \nonumber \\*
= & \sum_{\substack{R \in \mathcal R_N \ R \cap Q_J \ne \emptyset}}
\Bigl( \frac{1}{A_b(R)} \Bigl| \int_{R \cap Q_{J(2)}} F(w) \, dA_b(w) \Bigr| \Bigr)^2
\int_{R \cap Q_J} K\left(\frac{1-|z|}{|J|}\right) dA(z).
\end{align}
We treat the interior tiles $\mathcal R_N^{\mathrm{int}}=\{R_{N,j,k}:j\le2^N-2\}$ and the boundary tiles $\mathcal R_N^{\mathrm{bd}}=\{R_{N,2^N-1,k}:0\le k<2^N\}$ separately.

\smallskip
\noindent\emph{Case 1.1: Interior tiles ($R \in \mathcal R_N^{\mathrm{int}}$).}
For $R=R_{N,j,k}$ with $j\le2^N-2$, the bounds $(2^N-j-1)2^{-N}<1-|\zeta|^2\le(2^N-j)2^{-N}$ on $R$ show that the ratio $(1-|z|^2)/(1-|w|^2)$ is between $1/2$ and $2$ for $z,w\in R$. Thus $(1-|z|)/(1-|w|)\le4$. This lets us compare the weights at any two points in the same tile.

For $z\in R\cap Q_J$ and
$w\in R\cap Q_{J(2)}$, use $|J(2)|\le4|J|$ and
$(1-|w|)/|J(2)|\le1$ to obtain
\[
 K\left(\frac{1-|z|}{|J|}\right)
 \le\varphi_K(16)K\left(\frac{1-|w|}{|J(2)|}\right).
\]

Since $\varphi_K(16)<\infty$ by (K3) and
$A_b(R)\asymp(1-|w|^2)^bA(R)$ on $R$, Cauchy--Schwarz gives
\[
\begin{aligned}
 &\frac1{A_b(R)} \left(\int_{R\cap Q_J}K\left(\frac{1-|z|}{|J|}\right)dA(z)\right) \int_R|F_0(w)|^2dA_b(w)\\
 \lesssim &\int_R|F_0(w)|^2           K\left(\frac{1-|w|}{|J(2)|}\right)dA(w).
\end{aligned}
\]

\smallskip
\noindent\emph{Case 1.2: Boundary tiles ($R \in \mathcal R_N^{\mathrm{bd}}$).}
For $R=R_{N,2^N-1,k}$, apply weighted Cauchy--Schwarz with weight $K((1-|w|)/|J(2)|)$. The corresponding summand in \eqref{eq:p1_tile_sum_expand} is at most
\[
 \mathcal M(R)\int_R|F_0(w)|^2K\left(\frac{1-|w|}{|J(2)|}\right)dA(w),
\]
where
\[
 \mathcal M(R)=\frac{c_b^2}{A_b(R)^2}
 \left(\int_{R\cap Q_J}K\left(\frac{1-|z|}{|J|}\right)dA(z)\right)
 \left(\int_{R\cap Q_{J(2)}}
 \frac{(1-|w|^2)^{2b}}{K((1-|w|)/|J(2)|)}dA(w)\right).
\]
Recall that $A_b(R) = (b+1) \cdot 2^{-N} \int_{1-2^{-N}}^1 (1-s)^b \, ds = 2^{-(b+2)N}$. To bound $\mathcal{M}(R)$, we compare the tile size with the length of the testing arc.

\medskip
\noindent\emph{Subcase (i): $2^{-N} \le |J|$.}
Using $1-|z|\le1-|z|^2=t\le2^{-N}$ and the angular width $2^{-N}$ of $R$, we have
\[
 \int_{R\cap Q_J}K\left(\frac{1-|z|}{|J|}\right)dA(z)
 \le2^{-N}\int_0^{2^{-N}}K\left(\frac t{|J|}\right)dt
 \le2^{-2N}K\left(\frac{2^{-N}}{|J|}\right).
\]
Set $t=1-|w|^2=2^{-N}u$, $0<u\le1$.  Lemma~\ref{lem:general-lower-bound} applies \begin{equation*}
K\left( u \cdot \frac{2^{-N}}{|J(2)|} \right) \ge c_0 u^\alpha K\left(\frac{2^{-N}}{|J(2)|}\right) 
\end{equation*}
  with $s=2^{-N}/|J(2)|\le1$. 
Since $t/2\le1-|w|\le t$ and $t/|J(2)|\le1$, doubling and monotonicity
gives
$K((1-|w|)/|J(2)|)\ge C_K^{-1}K(t/|J(2)|)$. Hence
\begin{align*}
 &\int_{R\cap Q_{J(2)}}\frac{(1-|w|^2)^{2b}}
                       {K((1-|w|)/|J(2)|)}dA(w)\\
 \le & C_K2^{-(2b+2)N}
       \int_0^1\frac{u^{2b}}{K(u2^{-N}/|J(2)|)}du\\
 \le &\frac{C_K2^{-(2b+2)N}}{c_0K(2^{-N}/|J(2)|)}
                  \int_0^1u^{2b-\alpha}du
 =\frac{C_K2^{-(2b+2)N}}
        {c_0(2b-\alpha+1)K(2^{-N}/|J(2)|)},
\end{align*}
The integral is finite because $b>\max\{\alpha,1\}$ implies $2b-\alpha+1>0$. Also, $|J|\le|J(2)|\le4|J|$, so (K4) gives $$K(2^{-N}/|J|)\le C_K^2K(2^{-N}/|J(2)|).$$ Substitution into the expression for the multiplier gives
\begin{equation}\label{eq:multiplier_subcase1}
 \mathcal M(R)\le\frac{(b+1)^2C_K}{c_0(2b-\alpha+1)}
 \frac{K(2^{-N}/|J|)}{K(2^{-N}/|J(2)|)}\lesssim1.
\end{equation}

\medskip
\noindent\emph{Subcase (ii): $2^{-N} > |J|$.}
On $R\cap Q_J$, $1-|z|\le|J|<2^{-N}$, so
\[
 \int_{R\cap Q_J}K\left(\frac{1-|z|}{|J|}\right)dA(z)
 \le2^{1-N}\int_0^{|J|}K\left(\frac t{|J|}\right)dt
 \le2^{1-N}|J|K(1).
\]
Similarly, set $t=1-|w|=|J(2)|u$. Lemma~\ref{lem:general-lower-bound}
gives $K(u)\ge c_0K(1)u^\alpha$, and therefore
\begin{align*}
 &\int_{R\cap Q_{J(2)}}\frac{(1-|w|^2)^{2b}}
                       {K((1-|w|)/|J(2)|)}dA(w)\\
\le &2^{1-N}\int_0^{|J(2)|}\frac{(2t)^{2b}}{K(t/|J(2)|)}dt\\
\le &\frac{2^{2b+1-N}|J(2)|^{2b+1}}{c_0K(1)}
                 \int_0^1u^{2b-\alpha}du
 \lesssim2^{-N}|J|^{2b+1}.
\end{align*}
Here $|J(2)|\le4|J|$. Combining the two integral estimates with
$A_b(R)^2=2^{-2(b+2)N}$ yields
\begin{equation}\label{eq:multiplier_subcase2}
 \mathcal M(R)\lesssim
 \left(\frac{|J|}{2^{-N}}\right)^{2b+2}\le1.
\end{equation}

We have proved $\mathcal M(R)\lesssim1$ in both cases. The tiles are disjoint, so we can sum the interior and boundary estimates. The left-hand side of \eqref{eq:goal_local_part} is then at most
\[
 C\int_{Q_{J(2)}}|F(w)|^2K\left(\frac{1-|w|}{|J(2)|}\right)dA(w)\le C.
\]
This completes the proof of \eqref{eq:goal_local_part}.

\medskip
\noindent\textbf{The far-part estimate \eqref{eq:goal_tail_part}.}
We estimate the weighted $L^2(Q_J)$ norm of each $E_N^{(b)}F_n$. Once these bounds are summable, Minkowski's inequality will give the estimate for their sum.

Fix $z \in Q_J$, and let $R = R_{N,j,k} \in \mathcal R_N$ be the tile containing $z$. If $R \cap (Q_{J(n+1)} \setminus Q_{J(n)}) = \emptyset$, then $E_N^{(b)} F_n(z) = 0$. Otherwise, choose $w\in R\cap(Q_{J(n+1)}\setminus Q_{J(n)})$. The definition of the weighted average and Cauchy--Schwarz give
\begin{align*}
 |E_N^{(b)}F_n(z)|
 &\le\frac{c_b}{A_b(R)}
 \left(\int_{Q_{J(n+1)}}|F(w)|^2
              K\left(\frac{1-|w|}{|J(n+1)|}\right)dA(w)\right)^{1/2}\\
 &\quad\times\left(\int_{Q_{J(n+1)}}
        \frac{(1-|w|^2)^{2b}}{K((1-|w|)/|J(n+1)|)}dA(w)\right)^{1/2}.
\end{align*}
The first integral is at most $1$ because $\|F\|_{T_K}=1$. By Lemma~\ref{lem:tail-estimate}, the last factor is at most $C_b|J(n+1)|^{b+1}$.

We next use the fact that one tile meets both the original box and a distant annulus. For a nonempty annulus, $|J(n)|=2^n|J|<1$ and $n\ge2$. Since $w\notin Q_{J(n)}$, its angular coordinate is outside $J(n)$ or $1-|w|>|J(n)|$. In the first case, the normalized circular distance between the angular coordinates of $z$ and $w$ is at least $(|J(n)|-|J|)/2\ge3|J(n)|/8$. The two points belong to one angular tile of length $2^{-N}$, so $2^{-N}\gtrsim|J(n)|$. In the radial case, $|z|\ge1-|J|>3/4$, and
\[
2^{-N}\ge\bigl||z|^2-|w|^2\bigr|=(|z|+|w|)(|z|-|w|)\ge\tfrac34\bigl(|J(n)|-|J|\bigr)\ge\tfrac9{16}|J(n)|.
\]
Both cases therefore give $2^{-N}\gtrsim|J(n)|\asymp|J(n+1)|$.

For a boundary tile, $A_b(R)=2^{-(b+2)N}$. For an interior tile, $1-|\zeta|^2\ge2^{-N}$ and $A(R)=2^{-2N}$, which imply $$A_b(R)=c_b\int_R(1-|\zeta|^2)^bdA(\zeta)\ge c_b2^{-(b+2)N}.$$ Thus the weighted area of the tile has the lower bound
\begin{equation}\label{eq:AbR_exact_derivation}
 A_b(R)\ge2^{-(b+2)N}\gtrsim|J(n+1)|^b|J(n)|^2.
\end{equation}

Combining Cauchy--Schwarz with the lower bound \eqref{eq:AbR_exact_derivation}, we obtain
\begin{equation}\label{eq:p2_pointwise_bound_clean}
|E_N^{(b)} F_n(z)| \lesssim \frac{|J(n+1)|^{b+1}}{|J(n)|^2 |J(n+1)|^b} = \frac{|J(n+1)|}{|J(n)|^2} \asymp \frac{1}{\min\{2^n |J|, 1\}}, \quad \forall z \in Q_J.
\end{equation}
To pass from this pointwise bound to the required integral bound, use polar coordinates:
\[
 \int_{Q_J}K\left(\frac{1-|z|}{|J|}\right)dA(z)
 \le2|J|\int_0^{|J|}K\left(\frac t{|J|}\right)dt
 =2|J|^2\int_0^1K(u)du\lesssim|J|^2.
\]
Squaring \eqref{eq:p2_pointwise_bound_clean} and applying this estimate gives
\begin{equation}\label{eq:p2_slice_final_estimate}
 \int_{Q_J}|E_N^{(b)}F_n(z)|^2
         K\left(\frac{1-|z|}{|J|}\right)dA(z)
 \lesssim\frac{|J|^2}{(2^n|J|)^2}=2^{-2n}.
\end{equation}

Put $n_J=\lceil\log_2(1/|J|)\rceil$. The annular parts satisfy $F_n=0$ for $n\ge n_J$. Minkowski's inequality and \eqref{eq:p2_slice_final_estimate} give
\[
\begin{aligned}
 &\left(\int_{Q_J}\left|\sum_{n\ge2}E_N^{(b)}F_n(z)\right|^2
                  K\left(\frac{1-|z|}{|J|}\right)dA(z)\right)^{1/2}\\
 \le&\sum_{2\le n<n_J}
 \left(\int_{Q_J}|E_N^{(b)}F_n(z)|^2
                  K\left(\frac{1-|z|}{|J|}\right)dA(z)\right)^{1/2}
 \lesssim\sum_{n=2}^\infty2^{-n}\lesssim1.
\end{aligned}
\]
If the sum is empty, its value is zero. This establishes \eqref{eq:goal_tail_part}.
Combining \eqref{eq:goal_local_part} and \eqref{eq:goal_tail_part} completes the proof of Lemma~\ref{lem:Eb}.
\end{proof}

\subsection{Weighted Bergman projection and primitive representation}\label{sec:bergman}

For the fixed $b\ge b_K$, set
\[
 k_b(z,w)=\frac1{(1-z\bar w)^{b+2}},\qquad
 B_b\psi(z)=\int_\D\psi(w)k_b(z,w)\,dA_b(w).
\]
Here $k_b$ is the reproducing kernel of $A^2(dA_b)$.

We now construct an operator that recovers a function in $Q_K$ from its derivative. For $f\in Q_K$, the inclusion $Q_K\subset\mathcal B$ (see \cite{EWX,WZ}) gives
\begin{equation*}
\abs{f'(w)}\lesssim\frac{\norm{f}_{Q_K,*}}{1-\abs w^2}.
\end{equation*}
Hence, for $b\ge b_K>1$,
\begin{equation}\label{eq:deriv-L2}
\int_{\D}\abs{f'(w)}^2\,dA_b(w)\lesssim\norm{f}_{Q_K,*}^2\int_{\D}(1-\abs w^2)^{b-2}\,dA(w)<\infty,
\end{equation}
and therefore $f'\in A^2(dA_b)$, then the reproducing property gives $B_bf'=f'$. Integrating and applying Fubini's theorem yields
\begin{equation*}
f(z)-f(0)=\int_0^zf'(\xi)\,d\xi
=\int_0^z(B_bf')(\xi)\,d\xi
=\int_{\D}f'(w)\Bigl(\int_0^zk_b(\xi,w)\,d\xi\Bigr)\,dA_b(w).
\end{equation*}
This formula suggests integrating the Bergman kernel itself.

\begin{definition}\label{def:L}
For the fixed $b\ge b_K$, the integrated Bergman kernel is
\begin{equation*}
L_b(z,w)=\int_0^zk_b(\xi,w)\,d\xi
=\begin{cases}
\displaystyle\frac{(1-z\bar w)^{-(b+1)}-1}{(b+1)\bar w},&w\ne0,\\[6pt]
z,&w=0,
\end{cases}
\end{equation*}
and the reduced weighted Bergman projection is
\begin{equation*}
P_{b,0}\psi(z)=\int_{\D}\psi(w)\,L_b(z,w)\,dA_b(w),
\end{equation*}
defined for measurable functions $\psi$ for which the integral converges.
\end{definition}

The operator $P_{b,0}$ is well defined on $L^1(dA_b)$. As in \cite{HZ2}, for $z\in\D$ we have
\begin{equation}\label{eq:P-convergence}
\abs{P_{b,0}\psi(z)} \le\frac{\abs z}{(1-\abs z)^{b+2}}\,\norm{\psi}_{L^1(dA_b)}.
\end{equation}

\begin{lemma}\label{lem:Bb}
\cite[Lemma~2.3]{HLXZ} Let $K$ satisfy (K1)--(K4). For every $b\ge b_K$ the weighted Bergman projection $B_b$ acts boundedly on $T_K$, that is,
\begin{equation*}
\norm{B_b\psi}_{T_K}\lesssim \norm{\psi}_{T_K},\qquad \psi\in T_K .
\end{equation*}
\end{lemma}

\begin{lemma}\label{lem:P-derivative}
For $\psi\in L^1(dA_b)$ we have
\begin{equation*}
P_{b,0}\psi(z)=\int_0^zB_b\psi(\zeta)\,d\zeta,\qquad z\in\D,
\end{equation*}
and consequently $(P_{b,0}\psi)'=B_b\psi$ and $P_{b,0}\psi(0)=0$.
\end{lemma}

\begin{proof}
 Integrate the identity $$B_b\psi(\zeta)=\int_{\D}\psi(w)k_b(\zeta,w)\,dA_b(w)$$ from $0$ to $z$. This gives the formula, and differentiation gives the stated derivative. 
\end{proof}

To apply Lemma~\ref{lem:P-derivative} to functions in $T_K$, we must first check their integrability. The next lemma gives $T_K\hookrightarrow L^2(dA_b)\hookrightarrow L^1(dA_b)$ for large $b$.

\begin{lemma}\label{lem:embedding}
Let $K \in \WQ$ and let $\alpha=\log_2C_K$. If
$b\ge b_K$, then
\begin{equation*}
T_K\hookrightarrow L^2(dA_b)\hookrightarrow L^1(dA_b),
\end{equation*}
and there is a constant $C>0$ such that for every $\psi\in T_K$,
\begin{equation}\label{eq:chain-norm}
\norm{\psi}_{L^1(dA_b)}\le\norm{\psi}_{L^2(dA_b)}\lesssim \norm{\psi}_{T_K}.
\end{equation}
\end{lemma}
\begin{proof}
Taking $I=\T$ in \eqref{eq:tent-norm} gives
\begin{equation*}
\int_{\D}\abs{\psi(z)}^2K(1-\abs z)\,dA(z)\le\norm{\psi}_{T_K}^2.
\end{equation*}
Using $1-\abs z^2=(1-\abs z)(1+\abs z)\le2(1-\abs z)$ and Lemma~\ref{lem:general-lower-bound} with $t=1-\abs z$, we obtain
\[
\frac{(1-\abs z^2)^b}{K(1-\abs z)} \le\frac{2^b(1-\abs z)^b}{K(1-\abs z)} \le\frac{2^b}{c_0K(1)}\,(1-\abs z)^{b-\alpha} \le\frac{2^b}{c_0K(1)},
\]
because $b\ge b_K>\max\{\log_2 C_K,1\}=\max\{\alpha,1\}$. Therefore
\begin{align*}
\|\psi\|_{L^2(dA_b)}^2 
&= c_b \int_{\mathbb{D}} |\psi(z)|^2 K(1-|z|) \cdot \frac{(1-|z|^2)^b}{K(1-|z|)} \, dA(z) \\
&\le c_b \cdot \sup_{z \in \mathbb{D}} \frac{(1-|z|^2)^b}{K(1-|z|)} \int_{\mathbb{D}} |\psi(z)|^2 K(1-|z|) \, dA(z) \\
&\le c_b \cdot \frac{2^b}{c_0 K(1)} \int_{\mathbb{D}} |\psi(z)|^2 K(1-|z|) \, dA(z) \lesssim \|\psi\|_{T_K}^2.
\end{align*}

The Cauchy--Schwarz inequality gives $\norm{\psi}_{L^1(dA_b)}\le\norm{\psi}_{L^2(dA_b)}$. This proves \eqref{eq:chain-norm}.
\end{proof}

\begin{proposition}\label{prop:surjective}
Let $K \in \WQ$ and fix $b\ge b_K$. The map
\begin{equation*}
P_{b,0}\colon T_K\to\ringQ
\end{equation*}
is a bounded and surjective linear operator. More precisely,
\begin{equation*}
\norm{P_{b,0}\psi}_{Q_K,*}\lesssim \norm{\psi}_{T_K},\qquad \psi\in T_K,
\end{equation*}
and every $f\in\ringQ$ can be written as $f=P_{b,0}\psi$ with $\psi=f'$, for which $\norm{\psi}_{T_K}\asymp\norm{f}_{Q_K,*}$.
\end{proposition}

\begin{proof}
By Lemmas~\ref{lem:embedding} and \ref{lem:P-derivative}, we have $(P_{b,0}\psi)'=B_b\psi$ and $P_{b,0}\psi(0)=0$. Equation~\eqref{eq:QK-tent} and Lemma~\ref{lem:Bb} then give $$\|P_{b,0}\psi\|_{Q_K,*}\lesssim\|\psi\|_{T_K}.$$ This proves boundedness.

To prove surjectivity, fix $f\in\ringQ$ and choose $\psi=f'\in T_K\cap A^2(dA_b)$. This choice is allowed by \eqref{eq:QK-tent} and \eqref{eq:deriv-L2}. The reproducing property gives $B_bf'=f'$. We obtain $P_{b,0}f'=f$. Furthermore  $\|f'\|_{T_K}\asymp\|f\|_{Q_K,*}$.
\end{proof}

We now define the $Q_K$ packet matrix. Let $\mathcal R=\{\mathcal R_N\}_{N\ge0}$ be the polar dyadic resolution of $\D$. For a tile $R\in\mathcal R_N$, apply the reduced projection to its indicator function. The resulting analytic packet is
\begin{equation*}
\Psi^{(b)}_R(z)=P_{b,0}\mathbf 1_R(z)=\int_RL_b(z,w)\,dA_b(w).
\end{equation*}
For a finite positive Borel measure $\mu$ on $\D$ and each $N\ge0$, define the $Q_K$ packet matrix by
\begin{equation*}
\Gamma^{(b)}_{\mu,N}(R_1,R_2)=\int_{\D}\overline{\Psi^{(b)}_{R_1}(z)}\,\Psi^{(b)}_{R_2}(z)\,d\mu(z), \qquad R_1,R_2\in\mathcal R_N .
\end{equation*}
We check that its entries are finite whenever $\mu$ is finite. The identity $$L_b(z,w)=z\int_0^1(1-tz\bar w)^{-b-2}dt,$$ the critical weighted-kernel estimate, and Tonelli's theorem give
\[
 \int_\D|L_b(z,w)|dA_b(w)
 \lesssim_b\int_0^1\left(1+\log\frac1{1-t^2|z|^2}\right)dt
 \le C_b\int_0^1\left(1+\log\frac1{1-t^2}\right)dt<\infty.
\]
Thus $\|\Psi_R^{(b)}\|_{H^\infty}\le C_b$ with a bound independent of $R$.

\section{Embeddings }\label{sec:boundedness-I}

For coefficients $c=\{c_R\}_{R\in\mathcal R_N}$, define their discrete norm to be the tent norm of the function that has the corresponding constant value on each tile:
\begin{equation*}
\norm{c}_{X_{K,N}}^2 =\Bigl\lVert\sum_{R\in\mathcal R_N}c_R\,\mathbf 1_R\Bigr\rVert_{T_K}^2 =\sup_{I\subseteq\T}\sum_{R\in\mathcal R_N}\abs{c_R}^2\int_{R\cap Q_I}K\Bigl(\frac{1-\abs z}{\abs I}\Bigr)\,dA(z).
\end{equation*}

\begin{definition}\label{def:capacity}
Let $\mu$ be a finite positive Borel measure on $\D$. The dyadic capacity associated with the embedding problem is
\begin{equation*}
C^{(b)}_{K,\mathcal R}(\mu) =\sup_{N\ge0}\ \sup_{\norm{c}_{X_{K,N}}\le1} \int_{\D}\Bigl\lvert\sum_{R\in\mathcal R_N}c_R\,\Psi^{(b)}_R(z)\Bigr\rvert^2\,d\mu(z) =\sup_{N\ge0}\ \sup_{\norm{c}_{X_{K,N}}\le1}c^{*}\Gamma^{(b)}_{\mu,N}c.
\end{equation*}
\end{definition}

Here and in Definition~\ref{def:SDP}, \emph{capacity} means the quantity assigned to a measure by the stated optimization.

We compare this capacity with the best embedding constant:
\begin{equation}\label{eq:optimal}
C_{\ringQ}(\mu)=\sup\Bigl\{\int_{\D}\abs{f(z)}^2\,d\mu(z):f\in\ringQ,\ \norm{f}_{Q_K,*}\le1\Bigr\}.
\end{equation}
 $C_{\ringQ}(\mu)$ is the best constant for which the seminorm inequality \eqref{eq:trace-2} holds. The embedding \eqref{eq:trace} is bounded exactly when $\mu(\D)<\infty$ and $C_{\ringQ}(\mu)<\infty$.

\begin{theorem}\label{thm:capacity}
Let $K \in \WQ$ and let $\mu$ be a finite positive Borel measure on $\D$. The seminorm embedding inequality
\begin{equation}\label{eq:trace-2}
\int_{\D}\abs{f(z)}^2\,d\mu(z)\le C\norm{f}_{Q_K,*}^2,\qquad f\in\ringQ ,
\end{equation}
holds if and only if $C^{(b)}_{K,\mathcal R}(\mu)<\infty$. Moreover,
\begin{equation*}
C^{(b)}_{K,\mathcal R}(\mu)\asymp C_{\ringQ}(\mu).
\end{equation*}
\end{theorem}

\begin{proof}
Assume $C_*=C^{(b)}_{K,\mathcal R}(\mu)<\infty$ and fix $\psi\in T_K$. Write $E_N^{(b)}\psi=\sum_R a_R^{(N)}\mathbf1_R$. Lemma~\ref{lem:Eb} and the definition of capacity give
\[
 \int_\D|P_{b,0}E_N^{(b)}\psi|^2d\mu
 \le C_*\|a^{(N)}\|_{X_{K,N}}^2
 \lesssim C_*\|\psi\|_{T_K}^2.
\]
By Lemma~\ref{lem:embedding} and martingale convergence, $E_N^{(b)}\psi\to\psi$ in $L^1(dA_b)$. The kernel $L_b(z,\cdot)$ is bounded for each fixed $z\in\D$, so the projected functions converge pointwise. Fatou's lemma gives
\[
 \int_\D|P_{b,0}\psi|^2d\mu\lesssim C_*\|\psi\|_{T_K}^2.
\]
Now apply Proposition~\ref{prop:surjective} with $\psi=f'$. This proves $C_{\ringQ}(\mu)\lesssim C_*$.

For the reverse bound, take $\|c\|_{X_{K,N}}\le1$ and set $\psi_c=\sum_Rc_R\mathbf1_R$. Proposition~\ref{prop:surjective} gives
\[
 c^*\Gamma^{(b)}_{\mu,N}c
 =\|P_{b,0}\psi_c\|_{L^2(\mu)}^2
 \lesssim C_{\ringQ}(\mu)\|\psi_c\|_{T_K}^2
 \le C_{\ringQ}(\mu).
\]
Taking the supremum over the coefficients and then over the levels proves the reverse inequality.
\end{proof}

For each boundary arc $I\subseteq\T$, define the diagonal matrix:
\begin{equation*}
A^{(K)}_{I,N}=\diag\Bigl( \int_{R\cap Q_I}K\Bigl(\frac{1-\abs z}{\abs I}\Bigr)\,dA(z):R\in\mathcal R_N\Bigr).
\end{equation*}
For the zero-length arc, set $Q_{\varnothing}=\varnothing$ and $A^{(K)}_{\varnothing,N}=0$. The coefficient constraint $\norm{c}_{X_{K,N}}\le1$ can now be written as
\begin{equation}\label{eq:constraint}
c^{*}A^{(K)}_{I,N}c\le1,\qquad I\subseteq\T .
\end{equation}

\begin{definition}\label{def:SDP}
Let $\mu$ be a finite positive Borel measure on $\D$. Define the second dyadic capacity by
\begin{equation}\label{eq:capacity-D}
D^{(b)}_{K,\mathcal R}(\mu)=\sup_{N\ge0}\ \inf\Bigl\{ \sum_{m=1}^M\lambda_m:M\in\N,\ \lambda_m\ge0,\ I_m\subseteq\T,\ \Gamma^{(b)}_{\mu,N}\le\sum_{m=1}^M\lambda_mA^{(K)}_{I_m,N}\Bigr\}.
\end{equation}
The matrix inequality is taken in the cone $\mathcal S^{\#\mathcal R_N}_+$ of Hermitian positive semidefinite matrices. Thus the difference between the matrix on the right and the packet matrix must be positive semidefinite. Comparing eigenvalues alone is not enough. See \cite{VB} for background.
\end{definition}

\begin{remark}\label{rem:dyadic-arcs}
We can restrict the arc optimizations to the countable family $\mathcal D=\mathcal D^0\cup\mathcal D^{1/3}\cup\{\T\}$ without changing the quantities by more than a fixed factor. Here $\mathcal D^\tau$ consists of the intervals $[\tau+k2^{-n},\tau+(k+1)2^{-n}]$ modulo $1$, with $n\ge1$. Every arc $I$ is contained in an arc $I'\in\mathcal D$ satisfying $|I|\le|I'|\le6|I|$. Indeed, for $|I|<1/6$, choose $3|I|<2^{-n}\le6|I|$. The boundaries of the two grids are separated by $2^{-n}/3>|I|$, so at least one grid has no boundary point inside $I$. For larger arcs, choose $\T$.

We then have $Q_I\subseteq Q_{I'}$ and $A^{(K)}_{I,N}\le\varphi_K(6)A^{(K)}_{I',N}$. These comparisons give equivalent coefficient norms and domination capacities. Restricting further to a finite collection of arcs gives a finite-dimensional semidefinite program.
\end{remark}

\begin{theorem}\label{thm:SDP_detailed}
Let $K \in \WQ$ and let $\mu$ be a finite positive Borel measure on $\D$. For any fixed parameter $b\ge b_K$,
\begin{equation*}
C^{(b)}_{K,\mathcal R}(\mu)\asymp D^{(b)}_{K,\mathcal R}(\mu).
\end{equation*}
\end{theorem}

\begin{proof}
Fix $N\ge0$, set $d=\#\mathcal R_N$, and put
$\Gamma=\Gamma^{(b)}_{\mu,N}$. Let $\mathcal S_d^+$ be the cone of
$d\times d$ Hermitian positive semidefinite matrices.
 \begin{equation*}
S_d^+ = \bigl\{ A\in \mathbb{C}^{d\times d} : A^* = A,\ A \succeq 0 \bigr\}.
\end{equation*}

Let
$\mathcal M$ denote the family of all boundary arcs, including $\T$
and the empty arc. Write
\begin{align}
C^{(b)}_{K,\mathcal R,N}(\mu)
&:=\sup_{\norm{c}_{X_{K,N}}\le1}c^*\Gamma c,
\label{eq:primal-local}\\
\delta^{(b)}_{K,\mathcal R,N}(\mu)
&:=\inf\left\{\sum_{m=1}^M\lambda_m:
\begin{array}{l}
M\in\N,\quad\lambda_m\ge0,\quad I_m\in\mathcal M,\\
\displaystyle\Gamma\le\sum_{m=1}^M\lambda_m A^{(K)}_{I_m,N}
\end{array}\right\}.
\label{eq:dual-local}
\end{align}
If $\Gamma\le\sum_m\lambda_m A^{(K)}_{I_m,N}$ and
$\norm{c}_{X_{K,N}}\le1$, then
\[
c^*\Gamma c\le\sum_m\lambda_m c^*A^{(K)}_{I_m,N}c
\le\sum_m\lambda_m.
\]
Taking the supremum over $c$ and the infimum over finite dominations gives
\begin{equation}\label{eq:weak-duality-local}
C^{(b)}_{K,\mathcal R,N}(\mu)
\le\delta^{(b)}_{K,\mathcal R,N}(\mu).
\end{equation}

For the reverse bound, introduce a multiplier $X\in\mathcal S_d^+$.
The Lagrangian is
\[
\begin{aligned}
L(\lambda,X)
&=\sum_{m=1}^M\lambda_m
 -\Tr\left[X\left(\sum_{m=1}^M\lambda_m A^{(K)}_{I_m,N}-\Gamma\right)\right]\\
&=\Tr(\Gamma X)+\sum_{m=1}^M\lambda_m
       \left(1-\Tr(A^{(K)}_{I_m,N}X)\right).
\end{aligned}
\]
Its infimum over all finite choices of arcs and nonnegative coefficients
equals $\Tr(\Gamma X)$ precisely when
\[
\Tr(A^{(K)}_{I,N}X)\le1,\qquad I\in\mathcal M;
\]
otherwise that infimum is $-\infty$.
We justify equality of the primal and dual values by finite-family
SDP duality followed by compactness.

Every diagonal entry of $A^{(K)}_{\T,N}$ is positive, since
\[
\bigl(A^{(K)}_{\T,N}\bigr)_{R,R}
=\int_R K(1-|z|)\,dA(z)>0.
\]
Thus $A^{(K)}_{\T,N}>0$. For a finite family
$\mathcal F\subset\mathcal M$ containing $\T$, define
\[
\begin{aligned}
\delta_{\mathcal F}
&:=\inf\left\{\sum_{I\in\mathcal F}\lambda_I:
\lambda_I\ge0,\quad
\sum_{I\in\mathcal F}\lambda_I A^{(K)}_{I,N}\ge\Gamma\right\},\\
\mathcal K_{\mathcal F}
&:=\left\{X\in\mathcal S_d^+:
\Tr(A^{(K)}_{I,N}X)\le1\text{ for every }I\in\mathcal F\right\}.
\end{aligned}
\]
The finite primal problem is strictly feasible: take the coefficient
of $A^{(K)}_{\T,N}$ sufficiently large and all other coefficients positive.
The finite dual is strictly feasible at $X_\varepsilon=\varepsilon\mathrm I_d$,
where $0<\varepsilon<K(1)^{-1}$, because
\[
\Tr(A^{(K)}_{I,N}X_\varepsilon)
=\varepsilon\int_{Q_I}K\left(\frac{1-|z|}{|I|}\right)dA(z)
\le\varepsilon K(1)<1.
\]
For the empty arc the trace is zero. Finite-dimensional SDP duality
\cite[Section~3]{VB} therefore gives
\[
\delta_{\mathcal F}
=\max_{X\in\mathcal K_{\mathcal F}}\Tr(\Gamma X).
\]
Moreover,
\[
X\in\mathcal K_{\mathcal F}
\quad\Longrightarrow\quad
0\le\Tr X
\le\frac{\Tr(A^{(K)}_{\T,N}X)}
         {\lambda_{\min}(A^{(K)}_{\T,N})}
\le\frac1{\lambda_{\min}(A^{(K)}_{\T,N})}.
\]
Consequently every $\mathcal K_{\mathcal F}$ is a closed subset of the
same compact set $\mathcal K_{\{\T\}}$.
Since any finite family can be enlarged by $\T$ with zero coefficient,
\[
\delta^*:=\delta^{(b)}_{K,\mathcal R,N}(\mu)
=\inf_{\substack{\mathcal F\subset\mathcal M\text{ finite}\\\T\in\mathcal F}}
\delta_{\mathcal F}<\infty.
\]
The compact sets
\[
\mathcal K_{\mathcal F}\cap
\{X:\Tr(\Gamma X)\ge\delta^*\}
\]
are nonempty and have the finite-intersection property, since the set
corresponding to a finite union of arc families is contained in their
intersection. Hence there is an $X_*\ge0$ satisfying every arc
constraint and $\Tr(\Gamma X_*)\ge\delta^*$. Weak duality gives the
opposite inequality. Therefore
\begin{equation}\label{eq:zero-duality-gap}
\begin{aligned}
\delta^{(b)}_{K,\mathcal R,N}(\mu)
&=\max\left\{\Tr(\Gamma^{(b)}_{\mu,N}X):
\begin{array}{l}
X\in\mathcal S_d^+,\\
\Tr(A^{(K)}_{I,N}X)\le1\quad(I\in\mathcal M)
\end{array}\right\}.
\end{aligned}
\end{equation}
Write $d^*$ for this dual value; thus $d^*=\delta^*$.

Let $X \in \mathcal S_d^+$ be feasible in \eqref{eq:zero-duality-gap}. There are vectors $\{u_R\}_{R\in\mathcal R_N} \subset \C^d$ such that
\begin{equation}\label{eq:gram-decomp}
X(R_2, R_1) = \inn{u_{R_1}}{u_{R_2}}_{\C^d}, \qquad \forall R_1, R_2 \in \mathcal R_N.
\end{equation}
Set $x_R := \norm{u_R}_{\C^d} \ge 0$ and $v_R := x_R^{-1} u_R \mathbf{1}_{\{x_R > 0\}}$, so that $\norm{v_R}_{\C^d} \le 1$. By the definition of $A^{(K)}_{I,N}$ ,
\begin{equation}\label{eq:x-norm-constraint}
\Tr\bigl(A^{(K)}_{I,N} X\bigr) = \sum_{R\in\mathcal R_N} x_R^2 \bigl(A^{(K)}_{I,N}\bigr)_{R,R} \le 1 \quad (\forall I \in \mathcal M) \implies \norm{x}_{X_{K,N}} \le 1.
\end{equation}
Setting $\Lambda^{(b)}_{\mu,N}(R_1, R_2) :=  \Lambda^{(b)}_{\mu,N}(R_1, R_2) x_{R_1} x_{R_2}$ , so
 $\Lambda^{(b)}_{\mu,N}  \in \mathcal S_d^+$, then
\begin{equation}\label{eq:schur-expansion}
\Tr\bigl(\Gamma^{(b)}_{\mu,N} X\bigr) = \sum_{R_1, R_2 \in \mathcal R_N} \Lambda^{(b)}_{\mu,N}(R_1, R_2) \inn{v_{R_1}}{v_{R_2}}_{\C^d}.
\end{equation}
Applying the complex Grothendieck inequality (Grothendieck \cite{Grothendieck}, Pisier \cite[Theorem~1.1]{Pisier}) to $\Lambda^{(b)}_{\mu,N} \ge 0$ and $\{v_R\} \subset B_{\C^d}(0,1)$:
\begin{align}\label{eq:grothendieck-application}
\Bigl| \sum_{R_1, R_2 \in \mathcal R_N} \Lambda^{(b)}_{\mu,N}(R_1, R_2) \inn{v_{R_1}}{v_{R_2}}_{\C^d} \Bigr|
&\lesssim\sup_{\norm{\alpha}_\infty \le 1, \, \norm{\beta}_\infty \le 1} \bigl| \beta^* \Lambda^{(b)}_{\mu,N} \alpha \bigr|.
\end{align}
Because $\Lambda^{(b)}_{\mu,N} \ge 0$, the Cauchy--Schwarz inequality on the positive semidefinite sesquilinear form $(\alpha, \beta) \mapsto \beta^* \Lambda^{(b)}_{\mu,N} \alpha$ gives
\begin{equation}\label{eq:cs-diagonal-reduction}
\bigl| \beta^* \Lambda^{(b)}_{\mu,N} \alpha \bigr| \le \bigl(\alpha^* \Lambda^{(b)}_{\mu,N} \alpha\bigr)^{1/2} \bigl(\beta^* \Lambda^{(b)}_{\mu,N} \beta\bigr)^{1/2} \le \sup_{\norm{\alpha}_\infty \le 1} \alpha^* \Lambda^{(b)}_{\mu,N} \alpha.
\end{equation}
For any $\alpha \in \C^d$ with $\norm{\alpha}_\infty \le 1$, define $c_R := x_R \alpha_R$. Then
\begin{equation*}
c^* A^{(K)}_{I,N} c = \sum_{R\in\mathcal R_N} |\alpha_R|^2 x_R^2 \bigl(A^{(K)}_{I,N}\bigr)_{R,R} \le \sum_{R\in\mathcal R_N} x_R^2 \bigl(A^{(K)}_{I,N}\bigr)_{R,R} \le 1 \implies \norm{c}_{X_{K,N}} \le 1.
\end{equation*}
Consequently,
\begin{equation}\label{eq:primal-evaluation}
\alpha^* \Lambda^{(b)}_{\mu,N} \alpha = \sum_{R_1, R_2 \in \mathcal R_N} \overline{c_{R_1}}\Gamma^{(b)}_{\mu,N}(R_1,R_2)c_{R_2} = c^* \Gamma^{(b)}_{\mu,N} c \le C^{(b)}_{K,\mathcal R,N}(\mu).
\end{equation}
Combining \eqref{eq:schur-expansion}, \eqref{eq:grothendieck-application}, \eqref{eq:cs-diagonal-reduction}, and \eqref{eq:primal-evaluation} yields
\begin{equation*}
\Tr\bigl(\Gamma^{(b)}_{\mu,N} X\bigr) \lesssim \, C^{(b)}_{K,\mathcal R,N}(\mu).
\end{equation*}
Taking the supremum over all feasible $X \in \mathcal S_d^+$ in \eqref{eq:zero-duality-gap} gives
\begin{equation}\label{eq:upper-bound-local}
\delta^{(b)}_{K,\mathcal R,N}(\mu) \lesssim \, C^{(b)}_{K,\mathcal R,N}(\mu).
\end{equation}
Combining \eqref{eq:weak-duality-local} with \eqref{eq:upper-bound-local} yields
\begin{equation*}
C^{(b)}_{K,\mathcal R,N}(\mu) \le \delta^{(b)}_{K,\mathcal R,N}(\mu) \lesssim \, C^{(b)}_{K,\mathcal R,N}(\mu), \qquad \forall N \ge 0.
\end{equation*}
Taking $\sup_{N\ge0}$ across the chain establishes the equivalence in Theorem~\ref{thm:SDP_detailed}.
\end{proof}

\begin{theorem}\label{thm:embedding}
Let $K \in \WQ$ and let $\mu$ be a positive Borel measure on $\D$. The embedding
\begin{equation*}
\mathrm{id}\colon Q_K\to L^2(\mu)
\end{equation*}
is bounded if and only if $\mu(\D)<\infty$ and $D^{(b)}_{K,\mathcal R}(\mu)<\infty$.
\end{theorem}

\begin{proof}
If the embedding is bounded, testing the constant function $1$ gives $\mu(\D)<\infty$. Restricting to $\ringQ$ gives $C_{\ringQ}(\mu)<\infty$. Conversely, split a function as $f=f(0)+(f-f(0))$. Then
\[
 \int_\D|f|^2d\mu\le2\mu(\D)|f(0)|^2
       +2C_{\ringQ}(\mu)\|f\|_{Q_K,*}^2.
\]
Theorems~\ref{thm:capacity} and \ref{thm:SDP_detailed} identify the remaining embedding constant with the capacities.
\end{proof}

Compactness requires control of the integral near the boundary. For $\rho\in(0,1)$, write $S_\rho=\{z\in\D:\abs z>\rho\}$ for the outer part of the disk and $K_\rho=\{z\in\D:\abs z\le\rho\}$ for the compact inner disk.
The proof of the following lemma is standard, so we omit it.

\begin{lemma}\label{lem:compact-test}
The embedding $\mathrm{id}\colon \ringQ \to L^2(\mu)$ is compact if and only if the following condition holds: whenever $\{f_n\}$ is bounded in $\ringQ$ and $f_n \to 0$ uniformly on compact subsets of $\D$, we have $\norm{f_n}_{L^2(\mu)} \to 0$.
\end{lemma}

\begin{theorem}\label{thm:compact-embedding}
Let $K \in \WQ$ and let $\mu$ be a finite positive Borel measure on $\D$. Then
\begin{equation*}
\mathrm{id}\colon \ringQ \to L^2(\mu)
\end{equation*}
is compact if and only if
\begin{equation*}
C^{(b)}_{K,\mathcal R}(\mu) < \infty \quad\text{and}\quad \lim_{\rho\to1^-} C^{(b)}_{K,\mathcal R}(\mathbf 1_{S_\rho}\mu) = 0.
\end{equation*}
Equivalently, this holds if and only if
\begin{equation*}
D^{(b)}_{K,\mathcal R}(\mu) < \infty \quad\text{and}\quad
\lim_{\rho\to1^-} D^{(b)}_{K,\mathcal R}(\mathbf 1_{S_\rho}\mu) = 0.
\end{equation*}
\end{theorem}

\begin{proof}
Theorems~\ref{thm:capacity} and \ref{thm:SDP_detailed}, with constants
independent of the measure, give
\[
C_{\ringQ}(\mu)\asymp C^{(b)}_{K,\mathcal R}(\mu)
\asymp D^{(b)}_{K,\mathcal R}(\mu)
\]
for every finite positive measure $\mu$. Thus it suffices to use
$C_{\ringQ}$.

\medskip
\noindent\emph{Sufficiency.} Suppose $C_{\ringQ}(\mu)<\infty$ and $\lim_{\rho\to1^-}C_{\ringQ}(\mathbf 1_{S_\rho}\mu)=0$. By Lemma~\ref{lem:compact-test}, it suffices to show that for any sequence $\{f_n\}$ bounded in $\ringQ$ (say $\norm{f_n}_{Q_K,*} \le M$) with $f_n \to 0$ uniformly on compact subsets of $\D$, we have $\norm{f_n}_{L^2(\mu)} \to 0$.

Let $\epsilon>0$. Choose $\rho_0\in(0,1)$ such that $C_{\ringQ}(\mathbf 1_{S_{\rho_0}}\mu)<\frac{\epsilon^2}{2M^2}$. Then
\begin{equation}\label{eq:shell-bound}
\int_{S_{\rho_0}}\abs{f_n(z)}^2\,d\mu(z)
\le C_{\ringQ}(\mathbf 1_{S_{\rho_0}}\mu)\,\norm{f_n}_{Q_K,*}^2
< \frac{\epsilon^2}{2M^2} \cdot M^2 = \frac{\epsilon^2}{2}.
\end{equation}
On $K_{\rho_0}$, the uniform convergence $f_n\to 0$ and $\mu(K_{\rho_0})\le\mu(\D)<\infty$ yield  $N_0$ such that for all $n\ge N_0$,
\begin{equation}\label{eq:inner-bound}
\int_{K_{\rho_0}}\abs{f_n(z)}^2\,d\mu(z) \le \Bigl(\sup_{z\in K_{\rho_0}}\abs{f_n(z)}^2\Bigr)\mu(\D) \le\frac{\epsilon^2}{2}.
\end{equation}
Combining \eqref{eq:shell-bound} and \eqref{eq:inner-bound} gives
$\norm{f_n}_{L^2(\mu)}^2\le\epsilon^2$ for all $n\ge N_0$. By Lemma~\ref{lem:compact-test}, $\mathrm{id}\colon \ringQ \to L^2(\mu)$ is compact.

\medskip
\noindent\emph{Necessity.} Suppose that $\mathrm{id}\colon\ringQ\to L^2(\mu)$ is compact. Then it is bounded; by Theorem~\ref{thm:capacity},
$C^{(b)}_{K,\mathcal R}(\mu)<\infty$. Assume for a contradiction that $\lim_{\rho\to1^-}C^{(b)}_{K,\mathcal R}(\mathbf 1_{S_\rho}\mu)\ne0$.
Then there exist $\epsilon_0>0$ and $\rho_n\to1^-$ with
\begin{equation*}
C^{(b)}_{K,\mathcal R}(\mathbf 1_{S_{\rho_n}}\mu)\ge\epsilon_0,\qquad n\ge1.
\end{equation*}
By Theorem~\ref{thm:capacity}, $C_{\ringQ}(\mathbf 1_{S_{\rho_n}}\mu)\ge c\epsilon_0>0$, so there exist functions $f_n\in\ringQ$ with $\norm{f_n}_{Q_K,*}\le1$ and
\begin{equation}\label{eq:lower-mass}
\int_{S_{\rho_n}}\abs{f_n(z)}^2\,d\mu(z)\ge\frac{c\epsilon_0}{2}>0,\qquad n\ge1.
\end{equation}
By Montel's theorem, a subsequence converges uniformly on compact subsets to a function $f\in\ringQ$. Since the embedding is compact, Lemma~\ref{lem:compact-test} implies $\norm{f_n-f}_{L^2(\mu)}\to0$. By Minkowski's inequality,
\begin{equation}\label{eq:minkowski}
\Bigl(\int_{S_{\rho_n}}\abs{f_n}^2\,d\mu\Bigr)^{1/2}
\le\norm{f_n-f}_{L^2(\mu)}+\Bigl(\int_{S_{\rho_n}}\abs{f}^2\,d\mu\Bigr)^{1/2}.
\end{equation}
The first term on the right tends to $0$. Since $f\in L^2(\mu)$ and $1_{S_{\rho_n}}\to0$ pointwise in $\mathbb{D}$, dominated convergence gives \[\lim_{n\to\infty}\int_{S_{\rho_n}}\abs{f(z)}^2\,d\mu(z)=0.\]
Hence \eqref{eq:minkowski} forces $\int_{S_{\rho_n}}\abs{f_n}^2\,d\mu\to0$, which contradicts \eqref{eq:lower-mass}. Therefore
\[
\lim_{\rho\to1^-}C^{(b)}_{K,\mathcal R}(\mathbf 1_{S_\rho}\mu)=0.
\]
The proof is complete.
\end{proof}

Since the constant part has finite rank, the full embedding $Q_K\to L^2(\mu)$ is compact precisely when
\[
 \mu(\D)+D^{(b)}_{K,\mathcal R}(\mu)<\infty,
 \qquad\lim_{\rho\to1^-}D^{(b)}_{K,\mathcal R}(\mathbf1_{S_\rho}\mu)=0.
\]
 
We now consider arbitrary positive Borel measures. The logarithmic carleson box condition suggested by \eqref{eq:LW-nec-intro} is necessary for an embedding from $Q_K$ into $L^2(\mu)$, but it is not sufficient. We first prove necessity and then construct counterexamples. These examples concern general measures.

\begin{theorem}\label{thm:forward-K-carleson}
Let $K \in \WQ$ and let $\mu$ be a positive Borel measure on $\D$. If the embedding $\id\colon Q_K \hookrightarrow L^2(\mu)$ is bounded, then:
\begin{enumerate}[label=\rm{(\roman*)}]
\item There exists a constant $C_1=C_1(K)>0$ such that for every boundary arc $I \subsetneq \T$,
\begin{equation}\label{eq:log-decay-QI}
\mu(Q_I) \le \frac{C_1}{\left(\log\frac{2}{|I|}\right)^2} \|\id\|^2.
\end{equation}
\item The measure $\mu$ satisfies the condition
\begin{equation}\label{eq:log-weighted-single-box}
\sup_{I \subsetneq \T} \left(\log\frac{2}{|I|}\right)^2 \int_{Q_I} K\left(\frac{1-|z|}{|I|}\right) d\mu(z) \le C_1 K(1) \|\id\|^2 < \infty.
\end{equation}
\item In particular, $\mu$ is a vanishing $K$-Carleson measure:
\begin{equation}\label{eq:vanishing-limit}
\lim_{|I| \to 0} \int_{Q_I} K\left(\frac{1-|z|}{|I|}\right) d\mu(z) = 0.
\end{equation}
\end{enumerate}
\end{theorem}

\begin{proof}
Testing on $1$ gives $\mu(\D)\le\|\id\|^2$.
For $F_I$ from Lemma~\ref{lem:logarithmic-tests},
\[
 \log^2(2/|I|)\mu(Q_I)
 \le\int_\D|F_I|^2d\mu
 \le C_0(K)^2\|\id\|^2.
\]
This is (i). Since $K((1-|z|)/|I|)\le K(1)$ on $Q_I$, (ii) follows;
letting $|I|\to0$ gives (iii).
\end{proof}

The next theorem constructs a finite positive measure that satisfies both the  logarithmic carleson condition \eqref{eq:single-box-finite-test} and the stronger unweighted carleson condition \eqref{eq:single-box-unweighted}, but has no bounded  embedding. Thus these carleson box tests cannot replace the capacity criterion. The counterexample involving an analytic symbol is treated separately in Section~\ref{multi}.

\begin{theorem}\label{thm:single-box-counterexample}
Let $K \in \WQ$ and fix $b\ge b_K$. There exists a positive Borel measure $\mu$ on $\D$ for which $\mu$ satisfies
\begin{equation}\label{eq:single-box-finite-test}
\sup_{I\subseteq\T} \left(\log\frac{2}{|I|}\right)^2 \int_{Q_I} K\left(\frac{1-|z|}{|I|}\right) d\mu(z) < \infty,
\end{equation}
as well as the stronger condition
\begin{equation}\label{eq:single-box-unweighted}
\sup_{I\subseteq\T}\log^2(2/|I|)\,\mu(Q_I)<\infty,
\end{equation}
but both dyadic capacities are infinite:
\begin{equation*}
C^{(b)}_{K,\mathcal R}(\mu)=D^{(b)}_{K,\mathcal R}(\mu)=+\infty.
\end{equation*}
In particular, the embedding $Q_K \hookrightarrow L^2(\mu)$ is unbounded.
\end{theorem}

\begin{proof}
We put points at regular intervals on circles that move toward the boundary. For $k\ge1$, set
\begin{equation}\label{eq:atom-parameters}
 m_k=\lfloor2^{k/2}\rfloor,\qquad \ell_k=2^{-m_k},\qquad
 z_{k,j}=(1-\ell_k)e^{2\pi i j/2^k}\quad(1\le j\le2^k),
\end{equation}
and define them the measures
\begin{equation}\label{eq:mu-def-wk}
 \mu=\sum_{k\ge1}\sum_{j=1}^{2^k}
       \frac{\delta_{z_{k,j}}}{k^2 2^k K(1)}.
\end{equation}
The  measure is finite, since $\mu(\D)=K(1)^{-1}\sum_{k\ge1}k^{-2}<\infty$.

We first verify the stronger condition \eqref{eq:single-box-unweighted}. If $h=|I|<1/2$, any atom in $Q_I$ must satisfy $\ell_k\le h$. At most $2^kh+1$ point from level $k$ have arguments in $I$. Therefore
\begin{equation}\label{eq:box-decomp-eval}
 \mu(Q_I)\le\frac1{K(1)}\left(
 h\sum_{\ell_k\le h}\frac1{k^2}
 +\sum_{\ell_k\le h}\frac{2^{-k}}{k^2}\right).
\end{equation}
The first sum is controlled because $\sup_{0<h\le1}h\log^2(2/h)<\infty$. For the second sum, the condition $\ell_k\le h$ gives $\log(2/h)\le(m_k+1)\log2$ and $(m_k+1)^2 2^{-k}\le4$. Hence
\[
 \log^2(2/h)\sum_{\ell_k\le h}\frac{2^{-k}}{k^2}
 \le4\log^22\sum_{k\ge1}\frac1{k^2}<\infty.
\]
For large arcs, use $\mu(\D)<\infty$. This proves \eqref{eq:single-box-unweighted}. Since $K((1-|z|)/|I|)\le K(1)$ on $Q_I$, the weighted condition \eqref{eq:single-box-finite-test} follows as well.

To show that the embedding is unbounded, use the functions
\[
 F_N(z)=-\log(1-z^N),\qquad F_N'(z)=\frac{Nz^{N-1}}{1-z^N},
 \qquad B_K(N)=\sup_{1\le y\le N}yK(1/y)\ge K(1).
\]
We first bound their seminorms:
\begin{equation}\label{eq:FN-seminorm-bound}
 \|F_N\|_{Q_K,*}^2\lesssim_K B_K(N),\qquad N\ge2.
\end{equation}
For an arc $I$ of length $\ell$, periodicity and the Poisson integral give
\[
 \int_I\frac{dm(e^{i\theta})}{|1-r^Ne^{iN\theta}|^2}
 \le\frac{N\ell+2}{N(1-r^{2N})}.
\]
Indeed, $I$ meets at most $N\ell+2$ periods, and the integral over each period is $[N(1-r^{2N})]^{-1}$. Thus
\begin{equation}\label{eq:radial-reduction-raw}
 \int_{Q_I}|F_N'|^2K((1-|z|)/\ell)dA
 \lesssim N^2(\ell+N^{-1})\int_{1-\ell}^1
 \frac{r^{2N-1}}{1-r^{2N}}K((1-r)/\ell)dr.
\end{equation}
To estimate \eqref{eq:radial-reduction-raw}, we consider the inner and outer radial portions. The inner part is empty if $\ell<1/2$. Otherwise, for
$r\in[1-\ell,1/2]$, we have $1 - r^{2N} \ge 1 - 2^{-2N} \ge 3/4$, so
\begin{align}\label{eq:inner-radial-bound}
N^2 (\ell + N^{-1}) \int_{1-\ell}^{1/2} \frac{r^{2N-1}}{1 - r^{2N}} K\left(\frac{1-r}{\ell}\right) dr
&\le \frac{4}{3} N^2 (1 + 1) K(1) \int_0^{1/2} r^{2N-1} \, dr \notag\\
&= \frac{8}{3} N^2 K(1) \cdot \frac{(1/2)^{2N}}{2N} \notag\\
&= \frac{4}{3} K(1) \cdot N 2^{-2N} \lesssim K(1).
\end{align}
For $r \in (\max\{1/2, 1-\ell\}, 1)$, set $x := 1 - r \in (0, 1/2)$. By Bernoulli's inequality and convexity,
\begin{equation*}
1 - r^{2N} = 1 - (1-x)^{2N} \ge c_1 \min\{1, Nx\}, \qquad c_1 > 0,
\end{equation*}
and
\begin{equation*}
r^{2N-1} = (1-x)^{2N-1} \le \exp\left(-(2N-1)x\right) \le \exp(-Nx).
\end{equation*}
Therefore,
\begin{equation}\label{eq:kernel-fraction-bound}
\frac{r^{2N-1}}{1 - r^{2N}} \le \frac{C_2 e^{-Nx}}{\min\{1, Nx\}}.
\end{equation}
Substituting \eqref{eq:kernel-fraction-bound} into the integral and setting $y := N\ell > 0$ and $u := Nx \in (0, y)$,
\begin{align}\label{eq:u-substitution-integral}
N^2 (\ell + N^{-1}) \int_0^\ell \frac{e^{-Nx}}{\min\{1, Nx\}} K\left(\frac{x}{\ell}\right) dx
&= (y + 1) \int_0^y \frac{e^{-u}}{\min\{1, u\}} K\left(\frac{u}{y}\right) du.
\end{align}
For $0<y\le1$, this is bounded by $2\int_0^1K(t)dt/t\lesssim_K K(1)$, using (K3). For $y>1$, split at $u=1$. On $(0,1]$, we have
\[
 \int_0^1\frac{e^{-u}}uK(u/y)du
 \le K(1/y)\int_0^1\frac{\varphi_K(u)}u du
 \lesssim_K K(1/y).
\]
For $1\le u\le y$, the doubling property (K4) gives
$K(u/y)\le C_Ku^\gamma K(1/y)$, where $\gamma=\log_2C_K$. Consequently,
\[
 \int_1^y e^{-u}K(u/y)du
 \le C_KK(1/y)\int_1^\infty e^{-u}u^\gamma du
 \lesssim_K K(1/y).
\]
Now $(y+1)K(1/y)\le2B_K(N)$ proves \eqref{eq:FN-seminorm-bound}. In particular, $F_N$ belongs to $\ringQ$.

We now deduce a growth bound for $B_K(N)$ from (K3). By definition, $\varphi_K$ is non-decreasing and submultiplicative on $(0, 1]$:
\begin{equation*}
\varphi_K(s t) \le \varphi_K(s) \varphi_K(t), \qquad \forall s, t \in (0, 1].
\end{equation*}
Since $\int_0^1 \frac{\varphi_K(s)}{s} \, ds < \infty$, the monotonicity of $\varphi_K$ forces
\begin{equation*}
\lim_{s \to 0^+} \varphi_K(s) = 0.
\end{equation*}
Hence there exists $s_0 \in (0, 1)$ such that $q := \varphi_K(s_0) < 1$. Set
\begin{equation*}
\alpha_0 := \frac{\log(1/q)}{\log(1/s_0)} = \frac{\log q}{\log s_0} > 0.
\end{equation*}For $s_0^{n+1}<t\le s_0^n$, submultiplicativity gives
\[
 K(t)/K(1)\le\varphi_K(t)\le q^n
 \le s_0^{-\alpha_0}t^{\alpha_0}.
\]
Thus, with $\eta=\min(\alpha_0,1)>0$,
\begin{equation}\label{eq:BK-power-bound}
 B_K(N)\lesssim_K K(1)N^{1-\eta},\qquad
 \|F_N\|_{Q_K,*}^2\lesssim_K K(1)N^{1-\eta}.
\end{equation}

Set $N_k=2^k$ and $G_k=F_{N_k}$. We now compare the function values at the atoms with the seminorm bound. At every level-$k$ atom, $z_{k,j}^{N_k}=(1-\ell_k)^{N_k}$ and $1-(1-\ell_k)^{N_k}\le N_k\ell_k$. Hence
\[
 |G_k(z_{k,j})|\ge(m_k-k)\log2\gtrsim N_k^{1/2}
\]
for all sufficiently large $k$, uniformly in $j$. Normalize by setting $h_k=G_k/\|G_k\|_{Q_K,*}$. Equations~\eqref{eq:mu-def-wk} and \eqref{eq:BK-power-bound} give
\begin{equation}\label{eq:trace-capacity-divergence}
 C_{\ringQ}(\mu)\ge\int_\D|h_k|^2d\mu
 \gtrsim_K\frac{N_k^\eta}{k^2K(1)^2}
 \longrightarrow\infty.
\end{equation}
Thus the integrals are unbounded on the unit ball. Theorems~\ref{thm:capacity} and \ref{thm:SDP_detailed} imply $C^{(b)}_{K,\mathcal R}(\mu)=D^{(b)}_{K,\mathcal R}(\mu)=\infty$.
\end{proof}

\begin{corollary}\label{cor:radial-counterexample}
Let $K \in \WQ$. There exists a positive Borel measure $\mu$ on $\D$ such that $$\sup_{I \subsetneq \T} \left(\log\frac{2}{|I|}\right)^2 \int_{Q_I} K\left(\frac{1-|z|}{|I|}\right) d\mu(z) < \infty,$$ yet the embedding $Q_K \hookrightarrow L^2(\mu)$ is unbounded:
\begin{equation}\label{eq:radial-unbounded}
\sup_{f \in \ringQ, \, \|f\|_{Q_K} \le 1} \int_\D |f(z)|^2 \, d\mu(z) = +\infty.
\end{equation}
\end{corollary}

\begin{proof}
Place all point on one radius: let $z_n=1-2^{-n}$ and $\mu=\sum_{n\ge1}n^{-2}\delta_{z_n}$. Then $\mu(\D)=\pi^2/6$. For $2^{-(N+1)}<|I|\le2^{-N}$, with $N\ge1$, only point satisfying $n\ge N$ can belong to $Q_I$. Doubling gives
\[
 \int_{Q_I}K((1-|z|)/|I|)d\mu
 \le\frac{C_K}{N^2}\sum_{j\ge0}K(2^{-j})\lesssim_K N^{-2}.
\]
This sum is finite because $\int_0^1K(t)dt/t<\infty$. Since $\log(2/|I|)\le(N+2)\log2$, the required box test holds for small arcs.

Now take $f(z)=-\log(1-z)$. For $h=|I|$, integrating it gives
\[
 \int_{Q_I}|f'|^2K((1-|z|)/h)dA
 \le2\int_{1-h}^1\frac{K((1-r)/h)}{1-r}dr
 =2\int_0^1\frac{K(t)}t dt<\infty.
\]
Thus $f\in\ringQ$, but $$\int_\D|f|^2d\mu=\sum_{n\ge1}(n\log2)^2/n^2=\infty.$$ Normalize $f$ to obtain \eqref{eq:radial-unbounded}.
\end{proof}

\section{\texorpdfstring{Volterra integral operators on $Q_K$}{Volterra integral operators}}\label{sec:volterra-multipliers}

We apply the embedding criteria to the operators from the introduction. Recall the decomposition
\begin{equation}\label{eq:leibniz-decomp}
 M_gf=f(0)g(0)+T_gf+I_gf.
\end{equation}
For $a\in\D$, the measure associated with the symbol is
\begin{equation*}
d\mu_{g,a,K}(z)=\abs{g'(z)}^2\,K\bigl(1-\abs{\varphi_a(z)}^2\bigr)\,dA(z).
\end{equation*}
Equations~\eqref{eq:volterra-seminorm} and \eqref{eq:QK-seminorm} give a uniform bound for it:
\begin{equation}\label{eq:mass-identity}
\sup_{a\in\D}\mu_{g,a,K}(\D) =\sup_{a\in\D}\int_{\D}\abs{g'(z)}^2K\bigl(1-\abs{\varphi_a(z)}^2\bigr)\,dA(z) =\norm{g}_{Q_K,*}^2.
\end{equation}

\begin{theorem}\label{thm:volterra-bounded}
Let $K \in \WQ$ and $g\in H(\D)$. Then the Volterra operator $T_g$ is bounded on $Q_K$ if and only if
\begin{equation}\label{eq:volterra-condition}
g\in Q_K \quad\text{and}\quad
\sup_{a\in\D} D^{(b)}_{K,\mathcal R}(\mu_{g,a,K}) < \infty.
\end{equation}
Equivalently, this holds if and only if $g\in Q_K$ and $\sup_{a\in\D} C^{(b)}_{K,\mathcal R}(\mu_{g,a,K}) < \infty$.
\end{theorem}

\begin{proof}
By Theorems~\ref{thm:capacity} and~\ref{thm:SDP_detailed}, the two conditions
\[
 \sup_{a\in\D}C^{(b)}_{K,\mathcal R}(\mu_{g,a,K})<\infty,
 \qquad
 \sup_{a\in\D}D^{(b)}_{K,\mathcal R}(\mu_{g,a,K})<\infty
\]
are equivalent. It therefore
suffices to prove the theorem with $C_{\ringQ}$ in place of the capacities.

\medskip
\noindent\emph{Necessity.}
Suppose that $T_g$ is bounded on $Q_K$. Applying $T_g$ to the constant
function $1$ gives $T_g1=g-g(0)$, hence $g\in Q_K$. Now let $f\in\ringQ$ with
$\norm{f}_{Q_K,*}\le1$. Since $(T_gf)'=fg'$ and $T_gf(0)=0$, the identity
\eqref{eq:volterra-seminorm} gives, for each fixed $a\in\D$,
\[
 \int_\D\abs{f(z)}^2\,d\mu_{g,a,K}(z)
 \le\norm{T_gf}_{Q_K,*}^2
 =\norm{T_gf}_{Q_K}^2
 \le\norm{T_g}^2\norm{f}_{Q_K}^2
 =\norm{T_g}^2.
\]
Taking the supremum over all  $f$, and then over $a\in\D$, gives
\[
 \sup_{a\in\D}C_{\ringQ}(\mu_{g,a,K})\le\norm{T_g}^2<\infty .
\]

\medskip
\noindent\emph{Sufficiency.}
Suppose that $g\in Q_K$ and that
$M:=\sup_{a\in\D}C_{\ringQ}(\mu_{g,a,K})<\infty$. Write $f\in Q_K$ as
$f=f(0)+h$ with $h=f-f(0)\in\ringQ$. Then
\[
 T_gf=f(0)\bigl(g-g(0)\bigr)+T_gh .
\]
For the first term, $\abs{f(0)}\le\norm{f}_{Q_K}$ and
$\norm{g-g(0)}_{Q_K}=\norm{g}_{Q_K,*}$, so
\[
 \bigl\lVert f(0)\bigl(g-g(0)\bigr)\bigr\rVert_{Q_K}
 \le\norm{g}_{Q_K,*}\,\norm{f}_{Q_K}.
\]
For the second term, $(T_gh)(0)=0$, so
$\norm{T_gh}_{Q_K}=\norm{T_gh}_{Q_K,*}$. By \eqref{eq:volterra-seminorm} and
the definition of $C_{\ringQ}$,
\[
 \norm{T_gh}_{Q_K,*}^2
 =\sup_{a\in\D}\int_\D\abs{h(z)}^2\,d\mu_{g,a,K}(z)
 \le M\norm{h}_{Q_K,*}^2
 \le M\norm{f}_{Q_K}^2 .
\]
Combining the two estimates gives
\[
 \norm{T_gf}_{Q_K}
 \le\bigl(\norm{g}_{Q_K,*}+\sqrt M\,\bigr)\norm{f}_{Q_K}.
\]
So $T_g$ is bounded on $Q_K$.
\end{proof}

\begin{lemma}\label{lem:symbol-uniform-tails}
Let $K \in \WQ$ and let $g \in Q_K^{\log}$. Then the measures 
\[
d\mu_{g,a,K}(z) = |g'(z)|^2 K\bigl(1 - |\varphi_a(z)|^2\bigr)\, dA(z), \qquad a \in \D,
\]
satisfy
\begin{equation}\label{eq:symbol-uniform-tails}
\lim_{\rho \to 1^-} \sup_{a \in \D} \mu_{g,a,K}(S_\rho) = \lim_{\rho \to 1^-} \sup_{a \in \D} \int_{S_\rho} |g'(z)|^2 K\bigl(1 - |\varphi_a(z)|^2\bigr)\, dA(z) = 0,
\end{equation}
where $S_\rho = \{z \in \D : |z| > \rho\}$. In particular, \eqref{eq:symbol-uniform-tails} holds whenever $T_g$ is bounded on $Q_K$.
\end{lemma}

\begin{proof}
We first compare the M\"obius integral with the tent norm. For measurable $F\in T_K$, We claim that
\begin{equation}\label{eq:tent}
\sup_{a\in\mathbb{D}} \int_{\mathbb{D}} |F|^2 K (1 - |\varphi_a|^2) dA \lesssim \|F\|_{T_K}^2. 
\end{equation}

For $|a| \le 1/2$, the Möbius identity gives $1 - |\varphi_a(z)|^2 \le 8(1-|z|)$.
We have 
\[
\int_{\mathbb{D}} |F|^2 K (1-|\varphi_a|^2) dA \lesssim \varphi_K(8) \|F\|_{T_K}^2.
\]

For $|a|>1/2$, put $h=1-|a|$ and $n_*=\lceil\log_2(1/h)\rceil$. Let $J_k$ be the arc centered at $a/|a|$ with length $H_k=\min(2^kh,1)$, for $0\le k\le n_*$. Split the disk into
\[
 E_0=Q_{J_0},\qquad E_k=Q_{J_k}\setminus Q_{J_{k-1}}
 \quad(1\le k\le n_*).
\]
On $E_0$, we have $|1-\bar az|\ge h$. On $E_k$, with $k\ge1$, either $1-|z|>H_{k-1}\ge H_k/2$ or the angular distance from $a/|a|$ is at least a fixed multiple of $H_k$. The polar identity
\[
 |1-\bar az|^2=(1-|a||z|)^2
       +4|a||z|\sin^2((\arg z-\arg a)/2)
\]
then gives $|1-\bar az|\gtrsim H_k$. If $|z|<1/2$, the radial term gives this bound. Thus on $E_k$,
\[
 1-|\varphi_a(z)|^2
 \lesssim\frac{h}{H_k}\frac{1-|z|}{H_k}
 \le C2^{-k}\frac{1-|z|}{H_k}. 
\]
The estimate $h/H_k\le2^{1-k}$ remains valid when the last arc is the full circle. Since $E_k\subseteq Q_{J_k}$, we can sum to obtain
\[
 \int_\D|F|^2K(1-|\varphi_a|^2)dA
 \le\sum_{k=0}^{n_*}\varphi_K(C2^{-k})
       \int_{Q_{J_k}}|F|^2K((1-|z|)/H_k)dA
 \lesssim\|F\|_{T_K}^2.
\]
Condition (K3) gives $\sum_{k\ge0}\varphi_K(C2^{-k})<\infty$, this proves ~\eqref{eq:tent}.

Apply this estimate to $F=\mathbf1_{S_\rho}g'$. We obtain
\begin{equation}\label{eq:tail-measure-to-tent}
 \sup_a\mu_{g,a,K}(S_\rho)
 \lesssim\|\mathbf1_{S_\rho}g'\|_{T_K}^2.
\end{equation}
We now separate small and large testing arcs. Fix $0<\eta<1$. For $|I|<\eta$, the logarithmic tent norm bounds the tests by $\|g-g(0)\|_{Q_K^{\log}}^2/\log^2(2/\eta)$. For $|I|\ge\eta$, use $K((1-|z|)/|I|)\le\varphi_K(1/\eta)K(1-|z|)$. Together these bounds give
\begin{equation}\label{eq:truncated-symbol-tent-expanded}
 \|\mathbf1_{S_\rho}g'\|_{T_K}^2
 \le\frac{\|g-g(0)\|_{Q_K^{\log}}^2}{\log^2(2/\eta)}
 +\varphi_K(1/\eta)\int_{S_\rho}|g'|^2K(1-|z|)dA.
\end{equation}
The last integrand is integrable by the  logarithmic test. First let $\rho\to1^-$ and apply dominated convergence. Then let $\eta\to0^+$. Equation~\eqref{eq:tail-measure-to-tent} now gives the required uniform tail limit.

Finally, assume $T_g$ is bounded. Take $F_I$ from Lemma~\ref{lem:logarithmic-tests}. Its norm is uniformly bounded, and $|F_I|\ge\log(2/|I|)$ on $Q_I$. Hence
\[
 \log^2(2/|I|)\int_{Q_I}|g'|^2K((1-|z|)/|I|)dA
 \le\int_{Q_I}|(T_gF_I)'|^2K((1-|z|)/|I|)dA
 \lesssim\|T_g\|^2.
\]
We have $T_g1=g-g(0)\in Q_K$. We conclude that $g\in Q_K^{\log}$ and $\|g-g(0)\|_{Q_K^{\log}}\lesssim\|T_g\|$, which proves the last assertion.
\end{proof}

\begin{theorem}\label{thm:volterra-compact}
Let $K\in\WQ$ and $g \in H(\D)$. Then the Volterra operator $T_g$ is compact on $Q_K$ if and only if
\begin{equation}\label{eq:volterra-compact-D}
g \in Q_K \quad \text{and} \quad \lim_{\rho \to 1^-} \sup_{a \in \mathbb{D}} D^{(b)}_{K, \mathcal{R}}(\mathbf{1}_{S_\rho}\mu_{g,a,K}) = 0.
\end{equation}
Equivalently, this holds if and only if $g \in Q_K$ and $$\lim_{\rho \to 1^-} \sup_{a \in \mathbb{D}} C^{(b)}_{K, \mathcal{R}}(\mathbf{1}_{S_\rho}\mu_{g,a,K}) = 0.$$
\end{theorem}

\begin{proof}
\noindent\emph{Sufficiency.}
We first prove boundedness. Choose $\rho_0<1$ so that the supremum in \eqref{eq:volterra-compact-D} is finite at this radius. For $h\in\ringQ$, point evaluations on $K_{\rho_0}$ control the inner-disk integral, and Theorems~\ref{thm:capacity} and \ref{thm:SDP_detailed} control the outer integral. Thus
\[
\sup_a\int_\D |h|^2d\mu_{g,a,K}\lesssim\left(C_{\rho_0}\|g\|_{Q_K,*}^2+\sup_aD^{(b)}_{K,\mathcal R}(\mathbf1_{S_{\rho_0}}\mu_{g,a,K})\right)\|h\|_{Q_K,*}^2.
\]
Since $g\in Q_K$, the constant part of $f$ is controlled by $T_g1=g-g(0)$. Hence $T_g$ is bounded. We can now use the argument of Lemma~\ref{lem:compact-test} for $T_g$.

Let $\{f_n\}_{n=1}^\infty$ be a sequence in $Q_K$ satisfying $\sup_{n\ge1}\|f_n\|_{Q_K} \le M < \infty$ and $f_n \to 0$ uniformly on compact subsets of $\D$. Separate the value at the origin:
\[
f_n(z) = f_n(0) + h_n(z), \qquad \text{where } h_n(z) = f_n(z) - f_n(0) \in \ringQ.
\]
Then $h_n(0) = 0$, $$\|h_n\|_{Q_K,*} = \|f_n\|_{Q_K,*} \le M,$$  and $h_n \to 0$ uniformly on compact subsets of $\D$. Applying the operator gives
\[
T_g f_n(z) = f_n(0)\bigl(g(z) - g(0)\bigr) + T_g h_n(z).
\]
Because $f_n(0) \to 0$ as $n \to \infty$, 
\begin{equation}\label{eq:suff-constant-vanish}
\lim_{n\to\infty} \|f_n(0)(g - g(0))\|_{Q_K} = \lim_{n\to\infty} |f_n(0)| \|g\|_{Q_K,*} = 0.
\end{equation}
Furthermore, $(T_g h_n)(0) = 0$ gives $\|T_g h_n\|_{Q_K}^2 = \|T_g h_n\|_{Q_K,*}^2$. By the definition of the associated measure,
\[
\|T_g h_n\|_{Q_K,*}^2 = \sup_{a\in\D} \int_\D |(T_g h_n)'(z)|^2 K\bigl(1 - |\varphi_a(z)|^2\bigr)\,dA(z) = \sup_{a\in\D} \int_\D |h_n(z)|^2 \,d\mu_{g,a,K}(z).
\]
Fix $\rho\in(0,1)$ and split the integral:
\begin{equation}\label{eq:suff-split}
\int_\D |h_n(z)|^2 \,d\mu_{g,a,K}(z) = \int_{K_\rho} |h_n(z)|^2 \,d\mu_{g,a,K}(z) + \int_{S_\rho} |h_n(z)|^2 \,d\mu_{g,a,K}(z).
\end{equation}
On  $K_\rho$:
\begin{align}\label{eq:suff-Krho-bound}
\sup_{a\in\D} \int_{K_\rho} |h_n(z)|^2 \,d\mu_{g,a,K}(z) 
&\le \left(\sup_{z\in K_\rho} |h_n(z)|^2\right) \sup_{a\in\D} \mu_{g,a,K}(K_\rho) \notag \\
&\le \left(\sup_{z\in K_\rho} |h_n(z)|^2\right) \sup_{a\in\D} \int_\D |g'(z)|^2 K\bigl(1 - |\varphi_a(z)|^2\bigr)\,dA(z) \notag \\
&= \left(\sup_{z\in K_\rho} |h_n(z)|^2\right) \|g\|_{Q_K,*}^2.
\end{align}
On  $S_\rho$, Theorem~\ref{thm:capacity} gives
\begin{align}\label{eq:suff-Srho-bound}
\sup_{a\in\D} \int_{S_\rho} |h_n(z)|^2 \,d\mu_{g,a,K}(z) 
&= \sup_{a\in\D} \int_\D |h_n(z)|^2 \,d(\mathbf1_{S_\rho}\mu_{g,a,K})(z) \notag \\
&\lesssim \|h_n\|_{Q_K,*}^2 \sup_{a\in\D} D^{(b)}_{K,\mathcal R}(\mathbf1_{S_\rho}\mu_{g,a,K}) \notag \\
&\le M^2 \sup_{a\in\D} D^{(b)}_{K,\mathcal R}(\mathbf1_{S_\rho}\mu_{g,a,K}).
\end{align}
Combining \eqref{eq:suff-split}, \eqref{eq:suff-Krho-bound}, and \eqref{eq:suff-Srho-bound}, we obtain
\[
\|T_g h_n\|_{Q_K}^2 \lesssim \left(\sup_{z\in K_\rho} |h_n(z)|^2\right) \|g\|_{Q_K,*}^2 + M^2 \sup_{a\in\D} D^{(b)}_{K,\mathcal R}(\mathbf1_{S_\rho}\mu_{g,a,K}).
\]
Since $h_n \to 0$ uniformly on $K_\rho$, the inner-disk term disappears in the limit:
\[
\limsup_{n\to\infty} \|T_g h_n\|_{Q_K}^2 \lesssim M^2 \sup_{a\in\D} D^{(b)}_{K,\mathcal R}(\mathbf1_{S_\rho}\mu_{g,a,K}).
\]
Now let $\rho\to1^-$. Condition~\eqref{eq:volterra-compact-D} gives $\lim_{n\to\infty}\|T_g h_n\|_{Q_K}=0$. Together with \eqref{eq:suff-constant-vanish}, this proves $\lim_{n\to\infty}\|T_g f_n\|_{Q_K}=0$. Therefore $T_g$ is compact on $Q_K$.

\medskip
\noindent\emph{Necessity.}
Assume $T_g$ is compact on $Q_K$. Then $T_g$ is bounded, and Lemma~\ref{lem:symbol-uniform-tails} gives
\[\lim_{\rho\to1^-}\sup_{a\in\D}\mu_{g,a,K}(S_\rho) = 0.\]
We first prove tail decay for each fixed function. For $f\in Q_K$ and $r\in(0,1)$, consider the dilation $f_r(z) = f(rz)$. The Poisson formula $f(rz)=\int_0^{2\pi}P_r(\theta)f(e^{i\theta}z)\,d\theta/(2\pi)$, with $P_r(\theta)=(1-r^2)/|1-re^{i\theta}|^2$, combines with rotational invariance and Minkowski's inequality to give $\|f_r\|_{Q_K}\le\|f\|_{Q_K}$. The differences $f - f_r$ are bounded in $Q_K$ and tend to zero uniformly on compact subsets of $\D$ as $r\to1^-$. Compactness therefore gives $\lim_{r\to1^-}\|T_g(f - f_r)\|_{Q_K} = 0$. Since $f_r\in H^\infty$, we have
\begin{align*}
\left(\int_{S_\rho} |f|^2 \,d\mu_{g,a,K}\right)^{1/2} 
&\le \left(\int_{S_\rho} |f - f_r|^2 \,d\mu_{g,a,K}\right)^{1/2} + \left(\int_{S_\rho} |f_r|^2 \,d\mu_{g,a,K}\right)^{1/2} \\
&\le \|T_g(f - f_r)\|_{Q_K} + \|f_r\|_{H^\infty} \mu_{g,a,K}(S_\rho)^{1/2}.
\end{align*}
Take the supremum over $a\in\D$. First let $\rho\to1^-$, and then let $r\to1^-$. This gives
\begin{equation}\label{eq:single-f-decay}
\lim_{\rho\to1^-}\sup_{a\in\D}\int_{S_\rho}|f|^2\,d\mu_{g,a,K} \le \lim_{r\to1^-}\|T_g(f - f_r)\|_{Q_K}^2 = 0, \qquad \forall f\in Q_K.
\end{equation}
We next prove that the tail estimate is uniform over the unit ball. If \eqref{eq:volterra-compact-D} failed, Theorems~\ref{thm:capacity} and \ref{thm:SDP_detailed} would give $\epsilon_0 > 0$, radii $\rho_n\to1^-$, points $a_n\in\D$, and functions $f_n\in\ringQ$ with $\|f_n\|_{Q_K,*}\le1$ such that
\begin{equation}\label{eq:contra-hypo}
\int_{S_{\rho_n}} |f_n|^2 \,d\mu_{g,a_n,K} \ge \epsilon_0^2 > 0, \qquad \forall n\ge1.
\end{equation}
By Montel's theorem, a subsequence satisfies $f_n \to f$ uniformly on compact subsets of $\D$, where $f\in\ringQ$ and $\|f\|_{Q_K,*}\le1$. Thus $\{f_n - f\}$ is bounded in $Q_K$ and tends to zero uniformly on compact subsets of $\D$. Compactness of $T_g$ gives $\lim_{n\to\infty}\|T_g(f_n - f)\|_{Q_K} = 0$. The triangle inequality now yields
\begin{align*}
\epsilon_0 &\le \left(\int_{S_{\rho_n}} |f_n|^2 \,d\mu_{g,a_n,K}\right)^{1/2} \\
&\le \left(\int_{S_{\rho_n}} |f_n - f|^2 \,d\mu_{g,a_n,K}\right)^{1/2} + \left(\int_{S_{\rho_n}} |f|^2 \,d\mu_{g,a_n,K}\right)^{1/2} \\
&\le \|T_g(f_n - f)\|_{Q_K} + \left(\sup_{a\in\D}\int_{S_{\rho_n}} |f|^2 \,d\mu_{g,a,K}\right)^{1/2}.
\end{align*}
As $n\to\infty$, we have $\rho_n\to1^-$. The first term tends to zero by compactness, and the second by \eqref{eq:single-f-decay}. This contradicts $\epsilon_0>0$. Thus \eqref{eq:volterra-compact-D} is necessary, completing the proof of Theorem~\ref{thm:volterra-compact}.
\end{proof}

\section{Multipliers  } \label{multi}

%
%

We start with the multiplier test based on capacities.

\begin{theorem}\label{thm:multipliers}
Let $K \in \WQ$ and $g\in H(\D)$. Then the multiplication operator $M_g$ is bounded on $Q_K$, that is $g\in \mathcal{M}(Q_K)$, if and only if $g\in H^\infty \cap Q_K$ and
\begin{equation}\label{eq:multiplier-equiv}
\sup_{a\in\D} D^{(b)}_{K,\mathcal{R}}(\mu_{g,a,K}) < \infty.
\end{equation}
Equivalently, $\sup_{a\in\D} C^{(b)}_{K,\mathcal{R}}(\mu_{g,a,K}) < \infty$.
\end{theorem}

\begin{proof}
If $g\in H^\infty$, then $\|I_gf\|_{Q_K}\le\|g\|_{H^\infty}\|f\|_{Q_K,*}$. Theorem~\ref{thm:volterra-bounded} and the decomposition \eqref{eq:leibniz-decomp} therefore prove sufficiency.

Conversely, boundedness of $M_g$ gives $g=M_g1\in Q_K$ by testing the constant function. To bound the symbol pointwise, use the nonzero bounded evaluation functionals. They satisfy $M_g^*\delta_z=g(z)\delta_z$, which implies $|g(z)|\le\|M_g\|$ and hence $g\in H^\infty$. The identity $T_gf=M_gf-I_gf-f(0)g(0)$ now shows that $T_g$ is bounded. Theorem~\ref{thm:volterra-bounded} and the equivalence of the capacities give both stated criteria.
\end{proof}

\begin{corollary}
\label{cor:power-weight-multiplier-equivalence}
Let $0<p<1$, let $K(t)=t^p$ for $t\ge0$, and let $\mathcal R$ be the resolution from Section~\ref{sec:dyadic}. For $g\in H(\D)$, the following are equivalent:
\begin{enumerate}[label=\rm{(\roman*)}]
\item $g\in\mathcal M(Q_p)$.
\item $g\in H^\infty\cap Q_p$ and
\[
  \sup_{a\in\D}C^{(b)}_{K,\mathcal R}(\mu_{g,a,K})<\infty.
\]
\item $g\in H^\infty\cap Q_p$ and
\[
  \sup_{a\in\D}D^{(b)}_{K,\mathcal R}(\mu_{g,a,K})<\infty.
\]
\item $g\in H^\infty\cap Q_p^{\log}$; equivalently, $g\in H^\infty$ and
\begin{equation}\label{eq:power-weight-log-box}
  \sup_{I\subseteq\T}
  \frac{\log^2(2/|I|)}{|I|^p}
  \int_{Q_I}|g'(z)|^2(1-|z|^2)^p\,dA(z)<\infty.
\end{equation}
\end{enumerate}
Here
\[
  d\mu_{g,a,K}(z)
  =|g'(z)|^2(1-|\varphi_a(z)|^2)^p\,dA(z)
  =\frac{(1-|a|^2)^p}{|1-\bar a z|^{2p}}
    |g'(z)|^2(1-|z|^2)^p\,dA(z).
\]
\end{corollary}

\begin{proof}
The weight is admissible because $\varphi_K(s)=s^p$ and we may take
$\sigma=1$. Since $1-|z|\asymp1-|z|^2$, replacing $1-|z|^2$ by $1-|z|$ in
\eqref{eq:power-weight-log-box} changes the quantity only by a fixed factor.
Write $L$ for the resulting box quantity. Theorem~\ref{thm:multipliers} proves the equivalence of (i), (ii), and (iii).

For $g\in Q_p$, put $E=\sup_aC^{(b)}_{K,\mathcal R}(\mu_{g,a,K})$. The capacity theorem controls functions that vanish at the origin. Adding the estimate for constant functions gives
\begin{equation}\label{eq:power-weight-full-trace}
 \sup_a\int_\D|f|^2d\mu_{g,a,K}
 \lesssim(E+\|g\|_{Q_p,*}^2)\|f\|_{Q_p}^2.
\end{equation}
Test this inequality on $F_I$ from Lemma~\ref{lem:logarithmic-tests}. Using $(T_gF_I)'=F_Ig'$ and \eqref{eq:QK-box}, we obtain $L\lesssim E+\|g\|_{Q_p,*}^2$.

Conversely, if $L<\infty$, then $\|g\|_{Q_p,*}^2\lesssim L$ by \eqref{eq:QK-box}. Xiao's tent embedding theorem \cite[Theorem~1.1]{Xiao} gives
\begin{equation}\label{eq:power-weight-xiao-tent}
 \sup_J\frac1{|J|^p}\int_{Q_J}|f|^2|g'|^2(1-|z|^2)^pdA
 \lesssim L\|f\|_{Q_p}^2.
\end{equation}
For $|a|\le1/2$,  test $J=\T$ controls the M\"obius integral. For $|a|>1/2$, use the arcs $J_j$ of lengths $H_j=\min(2^j(1-|a|),1)$ from Lemma~\ref{lem:symbol-uniform-tails}. Splitting into the corresponding annuli and using $|1-\bar az|\gtrsim H_j$, we obtain
\[
 \begin{aligned}
 \int_\D|f|^2d\mu_{g,a,K}
 &\lesssim\sum_j\frac{(1-|a|)^p}{H_j^{2p}}
       \int_{Q_{J_j}}|f|^2|g'|^2(1-|z|^2)^pdA\\
 &\lesssim L\|f\|_{Q_p}^2
       \sum_j\left(\frac{1-|a|}{H_j}\right)^p
 \lesssim L\|f\|_{Q_p}^2.
 \end{aligned}
\]
since the sum is at most $\sum_{j=0}^\infty2^{-jp}+1$.
Thus the bound is uniform for all $a\in\D$. Restricting it
to $f(0)=0$ and using Theorems~\ref{thm:capacity} and
\ref{thm:SDP_detailed} gives
\[
  \sup_a C^{(b)}_{K,\mathcal R}(\mu_{g,a,K})\lesssim L,
  \qquad
  \sup_a D^{(b)}_{K,\mathcal R}(\mu_{g,a,K})\lesssim L.
\]
Hence (ii) and (iii) hold. This completes the proof.
\end{proof}

\begin{remark}\label{rem:power-weight-consistency}
The preceding result uses Xiao's tent embedding theorem. It applies to the specific measures $\mu_{g,a,K}$ generated by analytic symbols. Section~\ref{sec:boundedness-I} shows that a single-box criterion cannot characterize embeddings for arbitrary measures.
\end{remark}

\begin{theorem}\label{thm:multiplier-log-inclusion}
Let $K\in\WQ$. Then
\[
 \mathcal M(Q_K)\subseteq H^\infty\cap Q_K^{\log}.
\]
\end{theorem}
\begin{proof}
For $g\in\mathcal M(Q_K)$, the closed graph theorem makes $M_g$
bounded. The nonzero bounded point evaluations satisfy
$M_g^*\delta_z=g(z)\delta_z$, so
$\|g\|_{H^\infty}\le\|M_g\|$.
Since $$\|I_gf\|_{Q_K}\le\|g\|_{H^\infty}\|f\|_{Q_K,*},$$
\eqref{eq:leibniz-decomp} makes $T_g$ bounded. The last part of
Lemma~\ref{lem:symbol-uniform-tails} gives $g\in Q_K^{\log}$.
This proof does not use the capacity characterization.
\end{proof}

\subsection{A counterexample}

For the rest of the construction, fix $0<p<1$ and $\beta\ge1$ as in \eqref{eq:counterexample-weight}. We first verify that this weight is admissible. For $0<t<1$,
\[
 \frac{d}{dt}\log K(t)
 =\frac1t\left(p+\frac{\beta}{\log(e/t)}\right)>0.
\]
The two formulas agree at $t=1$, and $K$ is doubling. For $0<s\le1$ and $0<t\le1$,
\[
 \frac{K(st)}{K(t)}
 =s^p\left(\frac{\log(e/t)}{\log(e/t)+\log(1/s)}\right)^\beta
 \le s^p.
\]
Letting $t\downarrow0$ shows that the supremum defining $\varphi_K$ equals the stated power. For $s\ge1$, comparing the logarithmic factors gives
\[
 \varphi_K(s)=s^p\quad(0<s\le1),\qquad
 \varphi_K(s)\le s^p(1+\log s)^\beta\quad(s\ge1).
\]
Thus (K3) holds with $\sigma=1$ because $0<p<1$. For (K2), the substitution $x=\log(1/r)$ transforms the integral into $2\int_0^\infty K(x)e^{-2x}\,dx<\infty$. Conditions (K1) and (K4) follow from the definition, so $K\in\WQ$. These weights also satisfy conditions (2.1) and (2.2) of \cite{BW}, with $\sigma=1$. Thus the counterexample lies within those weight assumptions.

These weights still have geometric decay: $K(st)\le s^pK(t)$ for $0<s\le1$ and $t>0$. The difference from power weights lies in multiplication across scales. For $r,s\in(0,1]$,
\[
 \frac{K(rs)}{K(r)K(s)}
 =\left(\frac{\log(e/r)\log(e/s)}{\log(e/(rs))}\right)^\beta,
\]
and this ratio is unbounded as $r=s\downarrow0$. We choose the number of tree nodes to cancel the geometric factor. The remaining logarithmic factors then control the energy sums and the sums along a branch.

We will estimate the constructed functions using the box characterization \eqref{eq:QK-box}. For this construction, fix an integer $b\ge\max\{b_K,3\}$. Lemma~\ref{lem:Bb} gives the projection bound \eqref{eq:Bb-special}. Lemma~\ref{mx:logprojection} below gives the corresponding logarithmic bound. By Theorems~\ref{thm:capacity} and \ref{thm:SDP_detailed}, the final capacity conclusion will hold for every allowed $b$.

\begin{theorem}\label{mx:main}
For the weight \eqref{eq:counterexample-weight} there
are rational functions $f_N,g_N$, holomorphic on a neighborhood of
$\overline\D$, such that
\begin{align}
 f_N(0)&=0,&\quad \|f_N\|_{Q_K}&\le C,\label{mx:fbound}\\
 \norm{g_N}_{H^\infty}&\le \frac{C}{N+c},&
 \norm{g_N}_{Q_K^{\log}}&\le C,\label{mx:gbound}\\
 \|f_Ng_N\|_{Q_K}&\ge
 c_1\frac{\log(N+c)}{\sqrt{\sum_{m=0}^N(m+1)^{-\beta}}}
 -\frac{C}{N+c}.\label{mx:productbound}
\end{align}
These estimates hold for all sufficiently large $N$. The constants
$c>1,c_1>0,C<\infty$ are independent of $N$. In particular,
\begin{equation}\label{mx:strict}
 \mathcal M(Q_K)\subsetneq H^\infty\cap Q_K^{\log}.
\end{equation}
The lower bound in \eqref{mx:productbound} has order $\sqrt{\log N}$
when $\beta=1$, and order $\log N$ when $\beta>1$.
\end{theorem}

We write the projection bound from Lemma~\ref{lem:Bb} as
\begin{equation}\label{eq:Bb-special}
 \|B_b\psi\|_{T_K}\lesssim\|\psi\|_{T_K}.
\end{equation}
To construct the symbol with a bounded logarithmic norm, we need the following version of this estimate.

\begin{lemma}\label{mx:logprojection}
Let $K$ be the admissible weight defined by \eqref{eq:counterexample-weight}, and let $b \ge \max\{b_K, 3\}$ be an integer. Then the weighted Bergman projection $B_b$ is a bounded operator on the logarithmic tent space $T_K^{\log}$, that is,
\begin{equation}\label{eq:Bb-log-bounded}
\|B_b \psi\|_{T_K^{\log}} \lesssim \|\psi\|_{T_K^{\log}}, \qquad \forall \psi \in T_K^{\log},
\end{equation}
where the constant $C > 0$ depends only on $b, p, \beta$ and the structural constants of $K$.
\end{lemma}

\begin{proof}
Put $A=\|\psi\|_{T_K^{\log}}$. By definition, for every arc $J$,
\[
 \int_{Q_J}|\psi(w)|^2K((1-|w|)/|J|)dA(w)
 \le\frac{A^2}{\log^2(2/|J|)}.
\]
We will also use the finite integral
\begin{equation}\label{eq:moment-finiteness-detailed}
 \mathcal M_b(K)=\int_0^1\frac{u^{2b}}{K(u)}du
 =\int_0^1u^{2b-p}\log^\beta(e/u)du<\infty,
\end{equation}
whose convergence follows from $2b-p>-1$. For an arc $I$ of length $h$, set $\mathcal J(I)=\|B_b\psi\|_{L^2(Q_I,K((1-|z|)/h)dA)}$. Our goal is $\mathcal J(I)\lesssim A/\log(2/h)$ with a constant independent of $I$.

First assume that $\psi$ is bounded and compactly supported. For $h>1/16$, the ordinary estimate \eqref{eq:Bb-special} and $\|\psi\|_{T_K}\le A/\log2$ prove the claim, since $\log(2/h)\le\log32$.

Now suppose $h\le1/16$. Let $I_0$ be concentric with $I$, with $|I_0|=4h$, and put $\psi_0=\psi\mathbf1_{Q_{I_0}}$. We first estimate this local part. For a testing arc $J$ satisfying $|J|\le4h$,
\[
 \int_{Q_J}|\psi_0|^2K((1-|w|)/|J|)dA
 \le\frac{A^2}{\log^2(2/|J|)}
 \lesssim\frac{A^2}{\log^2(2/h)}.
\]
For $|J|>4h$, monotonicity gives $K((1-|w|)/|J|)\le K((1-|w|)/(4h))$ on the support of $\psi_0$. The integral is therefore bounded by the test on $I_0$, which is at most $$A^2/\log^2(2/(4h))\lesssim A^2/\log^2(2/h).$$ Taking the supremum over $J$ and applying \eqref{eq:Bb-special}, we obtain
\[
 \|B_b\psi_0\|_{L^2(Q_I,K((1-|z|)/h)dA)}
 \le\|B_b\psi_0\|_{T_K}\lesssim\frac{A}{\log(2/h)}.
\]

We estimate the remaining part on concentric annuli. Set $H_n=\min(2^n4h,1)$ for $0\le n\le n_*:=\lceil\log_2(1/(4h))\rceil$. Let $J_n$ be the concentric arcs with these lengths, with $J_0=I_0$, and define $\Omega_n=Q_{J_n}\setminus Q_{J_{n-1}}$ for $n\ge1$. These regions partition $\D\setminus Q_{I_0}$.

For $z\in Q_I$ and $w\in\Omega_n$, either $1-|w|>H_{n-1}\ge H_n/2$ or the angular distance from $Q_I$ is at least a fixed multiple of $H_n$. The polar identity gives $|1-z\bar w|\gtrsim H_n$, so
\[
 |B_b(\psi\mathbf1_{\Omega_n})(z)|
 \lesssim H_n^{-b-2}\int_{Q_{J_n}}|\psi(w)|(1-|w|^2)^bdA(w).
\]
Weighted Cauchy--Schwarz bounds the last integral by
\[
 \left(\int_{Q_{J_n}}|\psi|^2K((1-|w|)/H_n)dA\right)^{1/2}
 \left(\int_{Q_{J_n}}\frac{(1-|w|^2)^{2b}}{K((1-|w|)/H_n)}dA\right)^{1/2}.
\]
The first factor is at most $A/\log(2/H_n)$. For the second, use polar coordinates and $t=1-|w|$:
\[
\begin{aligned}
 \int_{Q_{J_n}}\frac{(1-|w|^2)^{2b}}{K((1-|w|)/H_n)}dA(w)
 &\le2^{2b+1}H_n\int_0^{H_n}\frac{t^{2b}}{K(t/H_n)}dt\\
 &=2^{2b+1}H_n^{2b+2}\mathcal M_b(K).
\end{aligned}
\]
It follows that
\[
 |B_b(\psi\mathbf1_{\Omega_n})(z)|
 \lesssim\frac{A}{H_n\log(2/H_n)},\qquad z\in Q_I.
\]
Also, $$\int_{Q_I}K((1-|z|)/h)dA\le2h^2\int_0^1K(u)du\lesssim h^2.$$ Minkowski's inequality therefore gives
\[
 \mathcal J(I)\lesssim\frac{A}{\log(2/h)}
                +A\sum_{n=1}^{n_*}\frac{h/H_n}{\log(2/H_n)}.
\]
It remains to bound this sum.  If $H_n\le\sqrt h$, then $\log(2/H_n)\ge\tfrac12\log(2/h)$ and $h/H_n=2^{-n-2}$. Hence
\[
 \sum_{H_n\le\sqrt h}\frac{h/H_n}{\log(2/H_n)}
 \le\frac{2}{\log(2/h)}\sum_{n\ge1}2^{-n-2}
 \lesssim\frac1{\log(2/h)}.
\]
For $H_n>\sqrt h$,  
\[
 \sum_{H_n>\sqrt h}\frac{h/H_n}{\log(2/H_n)}
 \le\frac1{\log2}\left(
   \sum_{\substack{1\le n<n_*\\2^n4h>\sqrt h}}2^{-n-2}+h\right)
 \lesssim\sqrt h\lesssim\frac1{\log(2/h)}.
\]
 The last comparison follows because $\sqrt h\log(2/h)$ is bounded on $(0,1/16]$. We have proved $\mathcal J(I)\lesssim A/\log(2/h)$ for every arc.

For a general   $\psi$, set $\psi_m=\psi\mathbf1_{\{|z|\le1-1/m,\ |\psi(z)|\le m\}}$, with $m\ge2$. Their logarithmic tent norms are at most $A$.  Cauchy--Schwarz  and \eqref{eq:moment-finiteness-detailed} give
\[
 \int_\D|\psi(w)|(1-|w|^2)^bdA(w)
 \le\left(\int_\D|\psi|^2K(1-|w|)dA\right)^{1/2}
     \left(\int_\D\frac{(1-|w|^2)^{2b}}{K(1-|w|)}dA\right)^{1/2}
 <\infty.
\]
Dominated convergence implies $\psi_m\to\psi$ in $L^1(dA_b)$. For $|z|\le r<1$, the kernel is bounded by $(1-r)^{-b-2}$ uniformly in $w$. Hence $B_b\psi_m\to B_b\psi$ uniformly on compact subsets. Apply Fatou's lemma on each $Q_I$ to the bounds for $\psi_m$, and then take the supremum over $I$. This proves \eqref{eq:Bb-log-bounded}.
\end{proof}

\begin{remark}\label{rem:logprojection-general}
The proof also applies to every $K\in\WQ$ and $b\ge b_K$:
Lemma~\ref{lem:Bb} supplies the ordinary bound, while
Lemma~\ref{lem:general-lower-bound} gives
$$\int_0^1u^{2b}/K(u)du\lesssim\int_0^1u^{2b-\log_2C_K}du<\infty.$$
\end{remark}

\subsection{A finite tree }

We first choose the intervals that form the tree, and later choose the coefficients of the analytic functions. Each node is a boundary interval. Its parent is the containing interval at the preceding level. A leaf is an interval at the final level $N$.

The numbers of intervals must meet two requirements. The weighted energy of the descendants of each node must be bounded. At the same time, the total weighted derivative mass on the leaves must stay bounded below. The formula below is chosen to satisfy both requirements.

All parameters are fixed before $N$ varies. Constants may depend on $p,\beta,b$ and the fixed choices $L,\eta,B,M,c,t_*$, but not on $N$, the levels, the testing arcs, or the coefficient arrays. We measure distances between boundary centers by normalized circular distance. For $\zeta=e^{i\theta}$ and $\xi=e^{i\phi}$, this distance is $(2\pi)^{-1}\min_{k\in\mathbb Z}|\theta-\phi+2\pi k|$.

Fix $L = 10^4(b+1)$ and $\eta = (100L)^{-1}$. Choose a sufficiently large integer $B$ and an integer $M$ such that
\begin{equation}\label{mx:parameters}
B^p < M < B/10, \qquad D_0 := \frac{B}{4M} \ge 100L.
\end{equation}
This is possible because $p < 1$. Next choose a fixed $c > 1$ satisfying
\begin{equation}\label{mx:cchoice}
p\log B + \frac{\beta+2}{c} < \log M.
\end{equation}
Set $t_* = 10^{-3}$, and choose a root arc $I_* \subseteq \T$ of length $s_* = t_*/L$. For each $N$ and level $0 \le j \le N$, define the desired number of intervals $P_{N,j}$ and the actual integer number $\mathcal{N}_j$ by
\begin{equation}\label{mx:population}
P_{N,j} = B^{pj} \left(\frac{j+c}{c}\right)^2 \left(\frac{N+c}{N-j+c}\right)^\beta, \qquad e_{N,j} = \left\lfloor\frac{\log P_{N,j}}{\log M}\right\rfloor, \qquad \mathcal{N}_j = M^{e_{N,j}}.
\end{equation}
Each factor in $P_{N,j}$ has a purpose. The factor $B^{pj}$ offsets the power decay of $K$. The factor $(j+c)^2$ matches the squared logarithm in the box norm. The factor involving the final depth keeps the total leaf mass bounded below. We verify these statements in \eqref{mx:ratios}, \eqref{mx:genergy}, and \eqref{mx:mass}. Replacing the desired numbers by the integer numbers changes the estimates only by fixed factors, independent of $N$.

The increments of $\log P_{N,j}$ satisfy
\[
\begin{aligned}
\log\frac{P_{N,j}}{P_{N,j-1}} &= p\log B + 2\log\left(1 + \frac{1}{j-1+c}\right) + \beta\log\left(1 + \frac{1}{N-j+c}\right) \\
&\le p\log B + \frac{\beta+2}{c} < \log M.
\end{aligned}
\]
Thus $e_{N,j} - e_{N,j-1} \in \{0, 1\}$ for every $1 \le j \le N$, and $e_{N,0} = 0$. Starting from $I_*$, let each interval at level $j-1$ have either one child at its center or $M$ equally spaced children in its central half. We use the first choice when $e_{N,j} - e_{N,j-1} = 0$ and the second when $1$. Every child has length $s_j = s_* B^{-j}$. At level $j$, denote the family of intervals by $\mathcal{L}_j$; it contains exactly $\mathcal{N}_j$ intervals.

More precisely, the $M$ children of a parent centered at $x$ and of length $s_{j-1}$ have centers $x + (k - \frac{M+1}{2}) \frac{s_{j-1}}{2M}$ for $1 \le k \le M$ in normalized angular coordinates. A single child has the same center $x$ as its parent.

\begin{figure}[!htbp]
\centering
\begin{adjustbox}{max width=\linewidth}
\begin{tikzpicture}[qkfig]
\begin{scope}
  \node[font=\footnotesize\bfseries] at (3,3.35)
    {(a) $e_{N,j}-e_{N,j-1}=0$: one child};
  \draw[qkmain,line width=1.4pt] (0,2.5)--(6,2.5);
  \node[above] at (3,2.55) {parent, length $s_{j-1}$};
  \draw[qkgray,densely dotted] (1.5,2.5)--(1.5,.8);
  \draw[qkgray,densely dotted] (4.5,2.5)--(4.5,.8);
  \draw[qkpath] (3,2.4)--(3,1.28);
  \draw[qkmain,line width=3pt] (2.75,1.02)--(3.25,1.02);
  \draw[qkdim] (2.75,.55)--(3.25,.55);
  \node[below,align=center] at (3,.52)
    {$s_j=s_{j-1}/B$\\the center is unchanged};
  \draw[qkbrace] (4.5,1.72)--(1.5,1.72);
  \node[below,fill=white,inner sep=1pt] at (3,1.61) {central half};
\end{scope}
\begin{scope}[xshift=8cm]
  \node[font=\footnotesize\bfseries] at (3,3.35)
    {(b) $e_{N,j}-e_{N,j-1}=1$: $M$ children};
  \draw[qkmain,line width=1.4pt] (0,2.5)--(6,2.5);
  \node[above] at (3,2.55) {parent, length $s_{j-1}$};
  \draw[qkgray,densely dotted] (1.5,2.5)--(1.5,.8);
  \draw[qkgray,densely dotted] (4.5,2.5)--(4.5,.8);
  \foreach \x in {1.8,2.4,3.6,4.2}{
     \draw[qklink] (3,2.4)--(\x,1.3);
     \draw[qkorange,line width=3pt] (\x-.13,1.02)--(\x+.13,1.02);
  }
  \node at (3,1.02) {$\cdots$};
  \node[below] at (1.8,.98) {$1$};
  \node[below] at (2.4,.98) {$2$};
  \node[below] at (3.6,.98) {$M-1$};
  \node[below] at (4.2,.98) {$M$};
  \draw[qkdim] (1.8,.24)--(2.4,.24);
  \node[below,align=center] at (3,.1)
    {center spacing $s_{j-1}/(2M)=2D_0s_j$};
\end{scope}
\end{tikzpicture}
\end{adjustbox}
\caption{One step of the finite-tree construction in \eqref{mx:population}--\eqref{mx:populationcompare}. All child intervals lie in the central half of their parent and have length \(s_j\). Only representative children are drawn in panel (b); the omitted children are indicated by dots. The drawing is not to scale, and the number of displayed segments is not \(M\).}
\label{fig:qk-branching}
\end{figure}

At each transition, every parent uses the same branching rule. Hence each interval at level $l$ has exactly $\mathcal{N}_j/\mathcal{N}_l$ descendants at level $j \ge l$. The actual and desired counts satisfy
\begin{equation}\label{mx:populationcompare}
M^{-1} P_{N,j} \le \mathcal{N}_j \le P_{N,j}, \qquad \text{and} \qquad M^{-1} \frac{P_{N,j}}{P_{N,l}} \le \frac{\mathcal{N}_j}{\mathcal{N}_l} \le M \frac{P_{N,j}}{P_{N,l}}.
\end{equation}
We next check separation. Distinct centers at level $j$ have normalized circular distance at least $D_0 s_j$. For siblings, their spacing is $\frac{s_{j-1}}{2M} = 2D_0 s_j > D_0 s_j$. For intervals with distinct parents $p_u \ne p_w \in \mathcal{L}_{j-1}$, the induction hypothesis $\operatorname{dist}_{\T}(\zeta_{p_u}, \zeta_{p_w}) \ge D_0 s_{j-1}$ and the central placement bound $\operatorname{dist}_{\T}(\zeta_u, \zeta_{p_u}) \le s_{j-1}/4$ give
\[
\operatorname{dist}_{\T}(\zeta_u, \zeta_w) \ge \operatorname{dist}_{\T}(\zeta_{p_u}, \zeta_{p_w}) - \frac{s_{j-1}}{4} - \frac{s_{j-1}}{4} \ge D_0 s_{j-1} - \frac{s_{j-1}}{2} = B\left(D_0 - \frac{1}{2}\right)s_j \ge D_0 s_j,
\]
Here we use $B \ge 10$ and $D_0 \ge 100L \ge 1$. We also need to know how far a descendant can move from its ancestor. For a descendant $v \in \mathcal{L}_k$ with $k > j$ of an interval $u \in \mathcal{L}_j$,
\[
\operatorname{dist}_{\T}(\zeta_v, \zeta_u) \le \sum_{m=j}^{k-1} \frac{s_m}{4} < \frac{s_j}{4} \sum_{i=0}^\infty B^{-i} = \frac{s_j}{4} \left(\frac{B}{B-1}\right) < \frac{s_j}{2},
\]
Thus every descendant center stays within $s_j/2$ of its ancestor center.

For each interval $v \in \mathcal{L}_j$, let $\zeta_v$ be its center and define
\begin{equation}\label{mx:atoms}
\begin{gathered}
t_j = L s_j = t_* B^{-j}, \qquad a_v = (1-t_j)\zeta_v, \\
A_v(z) = \left(\frac{t_j}{1 - \bar{a}_v z}\right)^{b+1}, \qquad E_v = \{z \in \D : |z - a_v| < \eta t_j\}.
\end{gathered}
\end{equation}
Every point of the disk $E_v$ has distance from the unit circle in $[(1-\eta)t_j, (1+\eta)t_j]$. We claim that all these disks are disjoint. For distinct $u,w\in\mathcal L_j$ at the same level,
\[
\begin{aligned}
|a_u - a_w| &= (1-t_j)|\zeta_u - \zeta_w| = 2(1-t_j)\sin(\pi \operatorname{dist}_{\T}(\zeta_u, \zeta_w)) \\
&\ge 4(1-t_j)D_0 s_j = 4(1-t_j)\left(\frac{D_0}{L}\right)t_j > 2\eta t_j,
\end{aligned}
\]
because $1-t_j > 1/2$, $D_0/L \ge 100$, and $\eta \le 10^{-6}$. For nodes at different levels $j \ne k$, say $j < k$, their radial ranges are disjoint because $$(1-\eta)t_j - (1+\eta)t_k \ge (1-\eta)B t_k - (1+\eta)t_k = [B(1-\eta) - (1+\eta)]t_k > 0.$$ This proves that the disks $\{E_v\}_{v \in \bigcup \mathcal{L}_j}$ are pairwise disjoint.

\subsection{Analytic construction and norm esimate}

Choose nonnegative level coefficients $d_j \ge 0$. The following rational function is holomorphic on a neighborhood of $\overline\D$:
\[
F_N(z) = \sum_{j=0}^N d_j \sum_{v\in\mathcal L_j} A_v(z).
\]
To represent its derivative by a Bergman projection, define the measurable function $\psi_N$ by
\begin{equation}\label{mx:psi}
\psi_N(w) = \sum_{j=0}^N \sum_{v\in\mathcal L_j} \frac{d_j \bar{a}_v t_j^{b+1}}{|E_v|(1-|w|^2)^b} \mathbf{1}_{E_v}(w), \qquad w \in \D,
\end{equation}
where $|E_v| = \int_{E_v} dA = \eta^2 t_j^2$. On a disk $E_v$ with $v \in \mathcal{L}_j$, we have $|a_v| = 1-t_j \asymp 1$ and $1-|w|^2 \asymp t_j$. Hence $|\psi_N(w)| \asymp \frac{d_j t_j^{b+1}}{\eta^2 t_j^2 t_j^b} \asymp d_j / t_j$, and integration over that disk gives $$\int_{E_v} |\psi_N(w)|^2 dA(w) \asymp d_j^2.$$

Applying the Bergman projection gives
\[
B_b \psi_N(z) = \sum_{j=0}^N \sum_{v\in\mathcal L_j} \frac{(b+1)d_j \bar{a}_v t_j^{b+1}}{|E_v|} \int_{E_v} \frac{dA(w)}{(1 - z\bar{w})^{b+2}}.
\]
The kernel is harmonic in the integration variable on each disk. Its mean value therefore satisfies
\[
\frac{1}{|E_v|} \int_{E_v} \frac{dA(w)}{(1 - z\bar{w})^{b+2}} = \frac{1}{(1 - z\bar{a}_v)^{b+2}}.
\]
Substitution gives the exact identity
\[
B_b \psi_N(z) = \sum_{j=0}^N d_j \sum_{v\in\mathcal L_j} A_v'(z) = F_N'(z).
\]
Since $\psi_N$ is bounded and compactly supported in $\D$, Lemma~\ref{lem:P-derivative} also gives $$P_{b,0}\psi_N(z) = \int_0^z B_b\psi_N(\zeta)\,d\zeta = F_N(z) - F_N(0).$$

Set $$k_m := K(B^{-m}) = \frac{B^{-pm}}{(1 + m\log B)^\beta} \asymp B^{-pm}(m+1)^{-\beta}.$$ We will prove the two tent-norm estimates
\begin{align}
\|\psi_N\|_{T_K}^2 &\lesssim \sup_{0\le l\le N} \sum_{j=l}^N \frac{\mathcal N_j}{\mathcal N_l} k_{j-l} d_j^2, \label{mx:tentestimate} \\
\|\psi_N\|_{T_K^{\log}}^2 &\lesssim \sup_{0\le l\le N} (l+c)^2 \sum_{j=l}^N \frac{\mathcal N_j}{\mathcal N_l} k_{j-l} d_j^2. \label{mx:logtentestimate}
\end{align}
Take any boundary arc $I \subseteq \T$ of length $h := |I| \in (0, 1]$. First determine which levels can meet its box. If $h < (1-\eta)t_N$, then $1-|z| \le h < (1-\eta)t_N$ on $Q_I$. No disk $E_v$ meets $Q_I$, so the integral is zero. If $h \ge (1-\eta)t_*$, set $l = 0$. Otherwise, choose the unique level $l \in \{1, \dots, N\}$ for which $(1-\eta)t_l \le h < (1-\eta)t_{l-1}$. A disk $E_v$ at a level $j < l$ satisfies $1-|w| \ge (1-\eta)t_j \ge (1-\eta)t_{l-1} > h$, and therefore $E_v \cap Q_I = \emptyset$. Thus only levels $j \ge l$ can contribute to the integral over $Q_I$.

For $l=0$, there is one root and $\mathcal N_0=1$. At most $\mathcal N_j=\mathcal N_j/\mathcal N_0$ disks at level $j$ can meet $Q_I$. For $l\ge1$, we count possible ancestor intervals. If a disk $E_v$ at level $j \ge l$ meets $Q_I$, its boundary center $\zeta_v$ is within normalized angular distance $C\eta t_j \le C\eta t_l$ of $I$. Its ancestor at level $l$ has center $\zeta_u$ satisfying $\operatorname{dist}_{\T}(\zeta_v, \zeta_u) < s_l/2$. Therefore all centers $\zeta_u \in \mathcal{L}_l$ of ancestors whose descendants meet $Q_I$ lie in an arc of length at most
\[
h + 2C\eta t_l + s_l \le (1-\eta)B L s_l + 2C\eta L s_l + s_l \le C B L s_l.
\]
The centers in $\mathcal{L}_l$ are separated by at least $D_0 s_l$. The number of possible ancestors is consequently at most $C(1 + \frac{CBLs_l}{D_0 s_l}) \lesssim C(1 + \frac{BL}{D_0}) \lesssim 1$, independently of $N$ and $I$. Each ancestor at level $l$ has exactly $\mathcal{N}_j/\mathcal{N}_l$ descendants at level $j$. Hence the number of disks at level $j \ge l$ meeting $Q_I$ satisfies
\[
\#\{v \in \mathcal L_j : E_v \cap Q_I \ne \emptyset\} \le C\left(1 + \frac{BL}{D_0}\right) \frac{\mathcal N_j}{\mathcal N_l} \lesssim \frac{\mathcal N_j}{\mathcal N_l}.
\]
We now estimate the weight on each contributing disk. For $z \in E_v \cap Q_I$ with $v \in \mathcal{L}_j$, we have $\frac{1-|z|}{h} \le \frac{(1+\eta)t_j}{(1-\eta)t_l} = (\frac{1+\eta}{1-\eta}) B^{-(j-l)}$. Doubling and the scaling of $K$ give $$K(\frac{1-|z|}{h}) \lesssim K(B^{-(j-l)}) = k_{j-l}.$$ Summing the contributions of all such disks yields
\[
\begin{aligned}
&\int_{Q_I} |\psi_N(w)|^2 K\left(\frac{1-|w|}{h}\right) dA(w) \\
 \le& \sum_{j=l}^N \sum_{\substack{v\in\mathcal L_j \\ E_v \cap Q_I \ne \emptyset}} \left(\sup_{w\in E_v} K\left(\frac{1-|w|}{h}\right)\right) \int_{E_v} |\psi_N|^2 \, dA
\lesssim \sum_{j=l}^N \frac{\mathcal N_j}{\mathcal N_l} k_{j-l} d_j^2.
\end{aligned}
\]
Taking the supremum over $I\subseteq\T$ proves \eqref{mx:tentestimate}.

For the logarithmic norm, suppose first that $l\ge1$. The choice of $l$ gives $h\asymp t_*B^{-l}$ and $\log(2/h)\asymp l+c$, with constants depending only on the fixed parameters. For $l=0$, we instead have $$\log(2/h)\le\log(2/((1-\eta)t_*))\lesssim c.$$ Thus $\log^2(2/h)\lesssim(l+c)^2$ in both cases. Multiplying the preceding box estimate by the squared logarithm proves \eqref{mx:logtentestimate}.

We can now transfer these estimates to the analytic functions. Combining \eqref{mx:tentestimate}, \eqref{eq:QK-box}, and \eqref{eq:Bb-special}, and using $F_N'=B_b\psi_N$, gives
\begin{equation}\label{mx:synthesis}
\|F_N - F_N(0)\|_{Q_K}^2 \lesssim \|\psi_N\|_{T_K}^2 \lesssim \sup_{0\le l\le N} \sum_{j=l}^N \frac{\mathcal N_j}{\mathcal N_l} k_{j-l} d_j^2.
\end{equation}
For the logarithmic norm, use \eqref{mx:logtentestimate}, Lemma~\ref{mx:logprojection}, and the same identity $F_N' = B_b\psi_N$. We obtain
\begin{equation}\label{mx:logsynthesis}
\|F_N - F_N(0)\|_{Q_K^{\log}}^2 = \|B_b \psi_N\|_{T_K^{\log}}^2 \lesssim \|\psi_N\|_{T_K^{\log}}^2 \lesssim \sup_{0\le l\le N} (l+c)^2 \sum_{j=l}^N \frac{\mathcal N_j}{\mathcal N_l} k_{j-l} d_j^2.
\end{equation}

\subsection{The two coefficient arrays}

We use one set of coefficients on all levels for the input function and a second set only on the leaves for the symbol. Set
\begin{equation}\label{mx:coefficients}
H_N = \sum_{m=0}^N (m+1)^{-\beta}, \qquad u_{N,j} = \frac{1}{(j+c)\sqrt{H_N}} \left(\frac{N-j+c}{N+c}\right)^\beta.
\end{equation}
Define the rational functions by
\begin{align}
f_N(z) &= \sum_{j=0}^N u_{N,j} \sum_{v\in\mathcal L_j} \bigl(A_v(z) - A_v(0)\bigr), \label{mx:fdef} \\
g_N(z) &= \frac{1}{N+c} \sum_{v\in\mathcal L_N} A_v(z). \label{mx:gdef}
\end{align}
Each center satisfies $|a_v| = 1-t_j < 1$. Hence the poles of $A_v(z)$ occur at $1/\bar{a}_v$, with modulus $(1-t_j)^{-1} > 1$, outside $\overline{\D}$. Subtracting $A_v(0)$ leaves the derivatives unchanged and gives $f_N(0) = 0$.

For $j \ge l$, the desired numbers of intervals satisfy
\[
\frac{P_{N,j}}{P_{N,l}} = B^{p(j-l)} \frac{(j+c)^2}{(l+c)^2} \left(\frac{N-l+c}{N-j+c}\right)^\beta.
\]
Multiply this ratio by $k_{j-l} \asymp B^{-p(j-l)}(j-l+1)^{-\beta}$ and use \eqref{mx:populationcompare}. This gives
\begin{equation}\label{mx:ratios}
\frac{\mathcal N_j}{\mathcal N_l} k_{j-l} \lesssim \frac{(j+c)^2}{(l+c)^2} \left(\frac{N-l+c}{N-j+c}\right)^\beta \frac{1}{(j-l+1)^\beta}.
\end{equation}
We first estimate the coefficients of the input function. Multiply \eqref{mx:ratios} by $$u_{N,j}^2 = \frac{1}{(j+c)^2 H_N} (\frac{N-j+c}{N+c})^{2\beta}$$ and sum over $j$:
\begin{align}\label{mx:fenergy}
\sum_{j=l}^N \frac{\mathcal N_j}{\mathcal N_l} k_{j-l} u_{N,j}^2 &\lesssim \frac{1}{(l+c)^2 H_N} \sum_{j=l}^N \frac{(N-l+c)^\beta (N-j+c)^\beta}{(N+c)^{2\beta} (j-l+1)^\beta} \notag \\
&\le \frac{1}{(l+c)^2 H_N} \sum_{j=l}^N \frac{1}{(j-l+1)^\beta} \le \frac{1}{(l+c)^2 H_N} \sum_{m=0}^N (m+1)^{-\beta} = \frac{1}{(l+c)^2},
\end{align}
Here we used $N-l+c \le N+c$ and $N-j+c \le N+c$. Taking the supremum over $l \in \{0, \dots, N\}$ in \eqref{mx:synthesis}, and using $f_N(0) = 0$, proves \eqref{mx:fbound}.

The symbol $g_N$ uses only the leaves: its coefficients are $d_N = \frac{1}{N+c}$ and $d_j = 0$ for $j < N$. Set $j = N$ in \eqref{mx:ratios} to obtain
\begin{equation}\label{mx:genergy}
\frac{\mathcal N_N}{\mathcal N_l} \frac{k_{N-l}}{(N+c)^2} \lesssim \frac{1}{(l+c)^2} \left(\frac{N-l+c}{c(N-l+1)}\right)^\beta \le \frac{1}{(l+c)^2},
\end{equation}
The last inequality uses $N-l+c \le c(N-l+1)$ for $c > 1$ and $N \ge l$. Substituting \eqref{mx:genergy} into \eqref{mx:logsynthesis} gives $\|g_N - g_N(0)\|_{Q_K^{\log}} \lesssim 1$.

It remains to bound $\|g_N\|_{H^\infty}$. Split $\D$ into the regions $\{|z| \ge 1/2\}$ and $\{|z| < 1/2\}$. For $|z| \ge 1/2$, put $\zeta_z = z/|z| \in \T$. The polar identity gives
\[
\begin{aligned}
|1 - \bar{a}_v z|^2
&= (1 - |a_v||z|)^2 + 4|a_v||z|\sin^2\left(\frac{\arg z - \arg a_v}{2}\right)\\
&\ge \frac{1}{4}t_N^2 + 4\left(\frac{1}{4}\right)\sin^2(\pi \operatorname{dist}_{\T}(\zeta_v, \zeta_z)),
\end{aligned}
\]
and hence $|1 - \bar{a}_v z| \ge c_0 (t_N + \operatorname{dist}_{\T}(\zeta_v, \zeta_z))$. The centers in $\mathcal{L}_N$ are separated by at least $D_0 s_N = (D_0/L)t_N$. Order them by angular distance from $\zeta_z$, separately on each side. Then
\[
\sum_{v\in\mathcal L_N} \left|\frac{t_N}{1 - \bar{a}_v z}\right|^{b+1} \le 2 \sum_{k=0}^\infty \left(\frac{t_N}{c_0(t_N + k (D_0/L)t_N)}\right)^{b+1} \lesssim \sum_{k=0}^\infty (1 + k)^{-(b+1)} \le C,
\]
For $|z| < 1/2$, we have $|1 - \bar{a}_v z| \ge 1 - |z| > 1/2$. Since $$\mathcal{N}_N \le (D_0 s_N)^{-1} = \frac{L}{D_0 t_N} \lesssim t_N^{-1},$$ the sum is at most $2^{b+1} \mathcal{N}_N t_N^{b+1} \lesssim t_N^b \le 1$. We have therefore proved
\begin{equation}\label{mx:pack}
\sup_{z\in\D} \sum_{v\in\mathcal L_N} \left|\frac{t_N}{1 - \bar{a}_v z}\right|^{b+1} \le C.
\end{equation}
Multiplication by $\frac{1}{N+c}$ gives $\|g_N\|_{H^\infty} \le \frac{C}{N+c}$. Finally, $|g_N(0)| \le \|g_N\|_{H^\infty} \lesssim 1$, so $$\|g_N\|_{Q_K^{\log}} \le |g_N(0)| + \|g_N - g_N(0)\|_{Q_K^{\log}} \le C.$$ This proves \eqref{mx:gbound}.

\subsection{Lower bounds on every leaf disk}

We now obtain lower bounds on the disjoint leaf disks $\{E_v\}_{v\in\mathcal L_N}$. These bounds will make the trace integral grow. Fix a leaf $v \in \mathcal{L}_N$ and a point $z = r e^{i\theta} \in E_v$. At each level $j \in \{0, \dots, N\}$, let $u \in \mathcal{L}_j$ be the ancestor of $v$, with center $a_u = (1-t_j)\zeta_u$ and $\zeta_u = e^{i\theta_u}$. The radial and angular distances satisfy $1-r \le (1+\eta)t_N \le (1+\eta)t_j$ and $$|\theta - \theta_u| \le 2\pi(s_j/2 + \eta t_N) \le 2\pi(1/2 + \eta L)s_j < 2\pi s_j = \frac{2\pi t_j}{L}.$$ We use these estimates to show that the ancestor terms have positive real parts.

Put $q := (1 - \bar{a}_u z)/t_j$. Its real part is
\[
\operatorname{Re} q = \frac{1 - (1-t_j)r\cos(\theta - \theta_u)}{t_j} = \frac{1-r}{t_j} + r + \frac{(1-t_j)r}{t_j}\bigl(1 - \cos(\theta - \theta_u)\bigr).
\]
Since $\cos(\theta-\theta_u)\le1$,
\[
\operatorname{Re}q\ge\frac{1-(1-t_j)r}{t_j}=1+\frac{(1-t_j)(1-r)}{t_j}\ge1.
\]
The inequality $1 - \cos x \le x^2/2$ gives the upper bound
\[
\operatorname{Re} q \le (1+\eta) + 1 + \frac{1}{2t_j} \left(\frac{2\pi t_j}{L}\right)^2 = 2 + \eta + \frac{2\pi^2 t_j}{L^2} < 3.
\]
For the imaginary part, $|\operatorname{Im} q| \le \frac{|\theta - \theta_u|}{t_j} \le \frac{2\pi}{L}$. Thus $|q| \le \sqrt{3^2 + (2\pi/L)^2} \le 4$ and $|\arg q| \le \arctan(2\pi/L) \le 2\pi/L$. Our choice $L = 10^4(b+1)$ makes the argument of the ancestor atom $A_u(z) = q^{-(b+1)}$ small:
\[
|\arg A_u(z)| = (b+1)|\arg q| \le (b+1) \frac{2\pi}{10^4(b+1)} = \frac{2\pi}{10^4} < \frac{\pi}{3}.
\]
Its real part therefore has a fixed positive lower bound:
\begin{equation}\label{mx:positiveatom}
\operatorname{Re} A_u(z) = |q|^{-(b+1)} \cos(\arg A_u(z)) \ge 4^{-(b+1)} \cos\left(\frac{\pi}{3}\right) = \frac{1}{2} 4^{-(b+1)}.
\end{equation}

\begin{figure}[!htbp]
\centering
\begin{adjustbox}{max width=\linewidth}
\begin{tikzpicture}[qkfig]
\begin{scope}
  \node[font=\footnotesize\bfseries] at (3,4.35) {(a) Ancestry identifies the main atom};
  \draw[qkmain,line width=2pt] (1.0,3.55)--(5.0,3.55);
  \node[above] at (3,3.6) {a common ancestor};
  \draw[qkpath] (3,3.43)--(1.7,2.18);
  \draw[qkother,-{Latex}] (3,3.43)--(4.7,2.18);
  \draw[qkmain,line width=2.8pt] (1.2,2.02)--(2.2,2.02);
  \draw[qkorange,line width=2.8pt] (4.2,2.02)--(5.2,2.02);
  \node[anchor=west,text=qkblue] at (2.35,2.02) {$u\supseteq v$};
  \node[anchor=west,text=qkorange] at (5.35,2.02) {$w\not\supseteq v$};
  \node[anchor=east] at (.65,2.02) {level $j$};
  \draw[qkpath,densely dashed] (1.7,1.83)--(1.7,.55);
  \draw[qkother] (4.7,1.83)--(4.7,.55);
  \draw[qkmain,line width=2.8pt] (1.5,.35)--(1.9,.35);
  \node[below] at (1.7,.28) {$v\in\mathcal L_N$};
  \node[anchor=east] at (.65,.35) {level $N$};
  \node[below,text=qkorange,align=center] at (4.7,.28)
    {a different\\subtree};
\end{scope}
\begin{scope}[xshift=8cm]
  \node[font=\footnotesize\bfseries] at (3,4.35) {(b) Same-level packing controls all other atoms};
  \draw[qkgray] (0,2.8)--(6.25,2.8);
  \foreach \x in {.6,1.9,4.5,5.8}{\node[qkoffpoint] at (\x,2.8) {};}
  \node[qkpoint] at (3.2,2.8) {};
  \node[above] at (3.2,2.9) {$\zeta_u$};
  \node[above,text=qkorange] at (1.9,2.9) {$k=1$};
  \node[above,text=qkorange] at (.6,2.9) {$k=2$};
  \node[above,text=qkorange] at (4.5,2.9) {$k=1$};
  \node[above,text=qkorange] at (5.8,2.9) {$k=2$};
  \draw[qklink] (3.38,1.83)--(3.38,2.7);
  \node[below,align=center] at (3.38,1.8) {$z/|z|$ near $\zeta_u$};
  \draw[qkdim] (3.38,3.65)--(5.8,3.65);
  \node[above,font=\scriptsize] at (4.65,3.67) {$\operatorname{dist}\ge D_0s_jk/2$};
  \node[qkaltbox,text width=6.3cm] at (3,.35)
    {$\displaystyle\sum_{w\not\supseteq v}|A_w(z)|
      \le 2\Bigl(\frac{2L}{D_0}\Bigr)^{b+1}
                        \sum_{k\ge1}k^{-(b+1)}$};
\end{scope}
\end{tikzpicture}
\end{adjustbox}
\caption{The separation used in \eqref{mx:positiveatom} and \eqref{mx:offpath}. In (a), dashed lines omit intermediate levels and branches. In (b), \(k\) counts the centers outside the ancestor path, in order on each side of \(z/|z|\). The factor \(2\) accounts for both sides. The inequality \(D_0/L\ge100\) makes the sum of all these other contributions small.}
\label{fig:qk-path-packing}
\end{figure}

We next control all other atoms at the same level. Regard the root arc as an interval on the real line. Order the centers $\zeta_w$, for $w\in\mathcal L_j\setminus\{u\}$, on each side of $\zeta_u$ by increasing distance from $\zeta_u$. If $w$ is the $k$-th center on one side, separation and $\operatorname{dist}_{\T}(\zeta_u,z/|z|)<s_j$ give $$\operatorname{dist}_{\T}(\zeta_w,z/|z|)\ge kD_0s_j-s_j\ge kD_0s_j/2.$$  Since $|a_w||z|>0.99$, the polar identity and the inequality $\sin(\pi d)\ge2d$ for $0\le d\le1/2$ imply $$|1-\bar a_wz|\ge2\operatorname{dist}_{\T}(\zeta_w,z/|z|)\ge kD_0t_j/(2L).$$ 
  Summing on both sides, we obtain
\begin{equation}\label{mx:offpath}
\sum_{\substack{w\in\mathcal L_j \\ w \ne u}} |A_w(z)| \le 2\left(\frac{2L}{D_0}\right)^{b+1} \sum_{k=1}^\infty k^{-(b+1)} \le 2(0.02)^{b+1} \left(\frac{4}{3}\right) < \frac{1}{4} 4^{-(b+1)}.
\end{equation}
Thus the other atoms cannot cancel the ancestor's positive contribution. Subtracting \eqref{mx:offpath} from \eqref{mx:positiveatom} gives
\begin{equation}\label{mx:sumpositive}
\operatorname{Re} \sum_{w\in\mathcal L_j} A_w(z) \ge \operatorname{Re} A_u(z) - \sum_{w \ne u} |A_w(z)| \ge \frac{1}{2} 4^{-(b+1)} - \frac{1}{4} 4^{-(b+1)} = \frac{1}{4} 4^{-(b+1)}.
\end{equation}

We must also account for the constants subtracted at the origin. Since $A_w(0)=t_j^{b+1}$,
\begin{equation}\label{mx:constantterm}
\sum_{j=0}^N u_{N,j} \mathcal N_j t_j^{b+1} \le \sum_{j=0}^N u_{N,j} P_{N,j} t_j^{b+1} = \frac{t_*^{b+1}}{c^2\sqrt{H_N}} \sum_{j=0}^N (j+c) B^{-(b+1-p)j} \lesssim \frac{1}{\sqrt{H_N}},
\end{equation}
By contrast, for $0 \le j \le N/2$ we have $\frac{N-j+c}{N+c} \ge 1/2$. Summing the coefficients over just these levels gives
\begin{equation}\label{mx:pathsum}
\sum_{j=0}^N u_{N,j} \ge \frac{2^{-\beta}}{\sqrt{H_N}} \sum_{j=0}^{\lfloor N/2 \rfloor} \frac{1}{j+c} \gtrsim \frac{\log(N+c)}{\sqrt{H_N}}.
\end{equation}
Combining \eqref{mx:sumpositive}, \eqref{mx:constantterm}, and \eqref{mx:pathsum}, we see that the positive contributions dominate the subtracted constants for all sufficiently large $N$. Hence
\begin{equation}\label{mx:flow}
|f_N(z)| \ge \operatorname{Re} f_N(z) \ge \left(\frac{1}{4} 4^{-(b+1)}\right) \sum_{j=0}^N u_{N,j} - \frac{C}{\sqrt{H_N}} \ge c_2 \frac{\log(N+c)}{\sqrt{H_N}}, \qquad \forall z \in E_v.
\end{equation}

We now bound the derivative of the symbol from below. On $E_v$, the bounds $|1 - \bar{a}_v z| \le 3t_N$ and $|a_v| \ge 1/2$ give
\[
|A_v'(z)| = \frac{(b+1)|a_v|t_N^{b+1}}{|1 - \bar{a}_v z|^{b+2}} \ge \frac{b+1}{2 \cdot 3^{b+2} t_N}.
\]
The same separation estimate bounds the remaining derivatives:
$$\sum_{w \ne v} |A_w'(z)| \le \frac{2(b+1)}{t_N} (\frac{2L}{D_0})^{b+2} \sum_{k=1}^\infty k^{-(b+2)} < \frac{b+1}{4 \cdot 3^{b+2} t_N}.$$ 
The triangle inequality then yields
\begin{equation}\label{mx:glow}
|g_N'(z)| \ge \frac{1}{N+c} \left(|A_v'(z)| - \sum_{w \ne v} |A_w'(z)|\right) \ge \frac{b+1}{4 \cdot 3^{b+2}(N+c)t_N}, \qquad \forall z \in E_v.
\end{equation}
At the final level $N$, the formula \eqref{mx:population} and the value of $K(t_N)$ give
\begin{equation}\label{mx:mass}
\frac{\mathcal N_N K(t_N)}{(N+c)^2} \asymp \frac{P_{N,N} K(t_N)}{(N+c)^2} = \frac{t_*^p}{c^{\beta+2}} \left(\frac{N+c}{1 + \log(1/t_*) + N\log B}\right)^\beta \asymp 1.
\end{equation}
This estimate keeps the total weighted mass from becoming small. The leaf disks $\{E_v\}_{v\in\mathcal L_N}$ are disjoint and lie in $\D$. Restricting the trace integral to their union and applying \eqref{mx:flow}, \eqref{mx:glow}, and \eqref{mx:mass}, we obtain
\begin{align}\label{mx:traceblow}
&\int_\D |f_N(z)|^2 |g_N'(z)|^2 K(1 - |z|^2) \, dA(z) \notag \\
\ge& \sum_{v\in\mathcal L_N} \int_{E_v} |f_N(z)|^2 |g_N'(z)|^2 K(1 - |z|^2) \, dA(z) \notag \\
\ge& \mathcal N_N \left(c_2^2 \frac{\log^2(N+c)}{H_N}\right) \left(\frac{c_3^2}{(N+c)^2 t_N^2}\right) \bigl(c_4 K(t_N)\bigr) \bigl(\eta^2 t_N^2\bigr) \notag \\
= &C_5 \frac{\log^2(N+c)}{H_N} \cdot \frac{\mathcal N_N K(t_N)}{(N+c)^2} \gtrsim \frac{\log^2(N+c)}{H_N}.
\end{align}
Finally, test the $Q_K$ seminorm at $a = 0$. The product rule and the reverse triangle inequality in $L^2(K(1-|z|^2)dA)$ give
\begin{align*}
\|f_N g_N\|_{Q_K,*} &\ge \left( \int_\D |(f_N g_N)'(z)|^2 K(1-|z|^2) \, dA(z) \right)^{1/2} \\
&\ge \left( \int_\D |f_N(z)|^2 |g_N'(z)|^2 K(1-|z|^2) \, dA(z) \right)^{1/2} \\
&\quad - \left( \int_\D |f_N'(z)|^2 |g_N(z)|^2 K(1-|z|^2) \, dA(z) \right)^{1/2} \\
&\ge c_1 \frac{\log(N+c)}{\sqrt{H_N}} - \|g_N\|_{H^\infty} \|f_N\|_{Q_K,*}.
\end{align*}
Using $\|g_N\|_{H^\infty} \le \frac{C}{N+c}$ and $\|f_N\|_{Q_K,*} \le C$ proves \eqref{mx:productbound}.

\subsection{A single counterexample and the capacity obstruction}

\begingroup

The finite pairs have bounded input norms, but their product norms tend to infinity. We first use this to disprove equality of the two multiplier classes. We then obtain one fixed analytic pair with an infinite trace. The Banach-space facts needed for both arguments were established in Section~\ref{sec:prelim}.

To finish the proof of Theorem~\ref{mx:main}, suppose that every $g\in H^\infty\cap Q_K^{\log}$ is a multiplier. Continuous point evaluations and the closed graph theorem show that each $M_g$ is bounded. Define the linear map
\[
 \Phi\colon H^\infty\cap Q_K^{\log}\longrightarrow\mathcal B(Q_K),
 \qquad \Phi(g)=M_g.
\]
We check that this map also has closed graph. Suppose $g_n\to g$ in the intersection norm and $M_{g_n}\to T$ in operator norm. For every $f\in Q_K$ and $z\in\D$,
\[
 (Tf)(z)=\lim_{n\to\infty}(M_{g_n}f)(z)
         =\lim_{n\to\infty}g_n(z)f(z)=g(z)f(z).
\]
Thus $T=M_g$, which proves that $\Phi$ has closed graph. Both spaces are Banach, so there is a constant $C_0$ such that
\begin{equation}\label{eq:closed-graph-bound}
 \begin{aligned}
 \|fg\|_{Q_K}
 &\le C_0\|f\|_{Q_K}\|g\|_{H^\infty\cap Q_K^{\log}},\\
 &\hspace{12mm}f\in Q_K,\quad g\in H^\infty\cap Q_K^{\log}.
 \end{aligned}
\end{equation}
However, the finite pairs in \eqref{mx:fbound}--\eqref{mx:productbound} satisfy
\[
 C_0C^2\ge\|f_Ng_N\|_{Q_K}
 \ge c_1\frac{\log(N+c)}{\sqrt{H_N}}-\frac{C}{N+c}
 \longrightarrow\infty,
\]
The divergence follows from $H_N\asymp\log N$ when $\beta=1$ and $H_N\asymp1$ when $\beta>1$. This contradiction and Theorem~\ref{thm:multiplier-log-inclusion} prove \eqref{mx:strict}.

\begin{remark}\label{rem:finite-multipliers}
Every individual $g_N$ is a multiplier. Indeed, if $g$ is holomorphic near $\overline\D$, the Bloch growth bound and $K(t)\le t^p$ give, for $h=|I|$,
\[
 \int_{Q_I}|fg'|^2K((1-|z|)/h)dA
 \lesssim\|g'\|_\infty^2\|f\|_{Q_K}^2
 h^{1-p}\int_0^h t^p\log^2(2/t)dt
 \lesssim\|g'\|_\infty^2\|f\|_{Q_K}^2.
\]
Therefore $$\|M_g\|\lesssim\|g\|_\infty+\|g'\|_\infty.$$ 
  What grows without bound in the finite construction is $\|M_{g_N}\|$, even though \eqref{mx:gbound} holds. Proposition~\ref{mx:fixedtrace} below gives one fixed pair with an infinite trace.
\end{remark}

\begin{proposition}\label{mx:fixedtrace}
For every weight $K$ of the form \eqref{eq:counterexample-weight}, there exist $g\in H^\infty\cap Q_K^{\log}$ and $f\in\ringQ$ such that
\begin{equation}\label{mx:fixedinfinite}
 \int_\D |f(z)|^2|g'(z)|^2K(1-|z|^2)\,dA(z)=\infty.
\end{equation}
In particular, $fg\notin Q_K$.
\end{proposition}

\begin{proof}
We first fix a symbol for which the trace integrals of the constructed input functions are unbounded. Put $d\nu_K(z)=K(1-|z|^2)\,dA(z)$ and define
\[
 R_N\colon \bigl(H^\infty\cap Q_K^{\log},\ \|\cdot\|_{H^\infty}+\|\cdot\|_{Q_K^{\log}}\bigr)
 \longrightarrow L^2(\nu_K),
 \qquad R_Nh=f_Nh'.
\]
Each operator $R_N$ is bounded. Indeed, \eqref{mx:intersection-embedding} gives
\[
 \|R_Nh\|_{L^2(\nu_K)}
 \lesssim\|f_N\|_{H^\infty}\|h-h(0)\|_{Q_K^{\log}}
 \le C\|f_N\|_{H^\infty}\|h\|_{H^\infty\cap Q_K^{\log}}.
\]
On the other hand, \eqref{mx:gbound} and \eqref{mx:traceblow} imply
\[
 \|R_N\|
 \ge\frac{\|f_Ng_N'\|_{L^2(\nu_K)}}
           {\|g_N\|_{H^\infty\cap Q_K^{\log}}}
 \gtrsim\frac{\log(N+c)}{\sqrt{H_N}}
 \longrightarrow\infty.
\]
By the uniform boundedness principle, there is a fixed $g\in H^\infty\cap Q_K^{\log}$ for which
\begin{equation}\label{eq:operator-blowup-fixed-g}
 \sup_N\|R_Ng\|_{L^2(\nu_K)}^2
 =\sup_N\int_\D|f_N|^2|g'|^2\,d\nu_K=\infty.
\end{equation}

We now show that some input has an infinite trace for this fixed symbol. Suppose instead that $hg'\in L^2(\nu_K)$ for every $h\in\ringQ$. Then the map
\[
 S_g\colon\ringQ\longrightarrow L^2(\nu_K),
 \qquad S_gh=hg',
\]
is defined on the whole space. We check that its graph is closed. Suppose $h_k\to h$ in $\ringQ$ and $S_gh_k\to u$ in $L^2(\nu_K)$. On each compact disk $K_r$,
\[
 \begin{aligned}
 \|u-hg'\|_{L^2(K_r,\nu_K)}
 &\le\|u-h_kg'\|_{L^2(\nu_K)}\\
 &\quad+\sup_{z\in K_r}|h_k(z)-h(z)|\,
              \|g'\|_{L^2(\nu_K)}
 \longrightarrow0.
 \end{aligned}
\]
Here $\|g'\|_{L^2(\nu_K)}<\infty$ follows from \eqref{mx:intersection-embedding}. Thus $u=hg'$ almost everywhere on $\D$, so $S_g$ has closed graph. The space $\ringQ$ is Banach with norm $\|\cdot\|_{Q_K,*}$. The closed graph theorem would therefore imply
\[
 \sup_N\|R_Ng\|_{L^2(\nu_K)}
 =\sup_N\|S_gf_N\|_{L^2(\nu_K)}
 \le\|S_g\|\sup_N\|f_N\|_{Q_K,*}<\infty,
\]
This contradicts \eqref{eq:operator-blowup-fixed-g}. Hence $fg'\notin L^2(\nu_K)$ for some $f\in\ringQ$, proving \eqref{mx:fixedinfinite}.

Finally, the other term in the product derivative is integrable:
\[
 \|f'g\|_{L^2(\nu_K)}
 \le\|g\|_{H^\infty}\|f\|_{Q_K,*}<\infty.
\]
If $fg\in Q_K$, the test at $a=0$ would give $(fg)'\in L^2(\nu_K)$. Subtracting the integrable term would then give $fg'=(fg)'-f'g\in L^2(\nu_K)$, a contradiction.
\end{proof}

\begin{remark}\label{rem:quantifiers-fixedtrace}
Every integral in \eqref{eq:operator-blowup-fixed-g} is finite, although their supremum is infinite. This alone does not give a single $f$ with an infinite trace. The closed graph argument supplies that function. The assertion is only that such a function exists; for example, the trace of $f\equiv1$ is finite.
\end{remark}

\begin{corollary}\label{mx:capacity}
For every weight $K$ defined by \eqref{eq:counterexample-weight}, the symbol $g$ in Proposition~\ref{mx:fixedtrace} satisfies, for the polar dyadic resolution $\mathcal R$ of Section~\ref{sec:dyadic},
\begin{equation}\label{mx:capacity-zero}
 C^{(b)}_{K,\mathcal R}(\mu_{g,0,K})
 =D^{(b)}_{K,\mathcal R}(\mu_{g,0,K})=\infty.
\end{equation}
Consequently,
\[
 \sup_{a\in\D}C^{(b)}_{K,\mathcal R}(\mu_{g,a,K})
 =\sup_{a\in\D}D^{(b)}_{K,\mathcal R}(\mu_{g,a,K})=\infty.
\]
\end{corollary}

\begin{proof}
Fix the pair $(f,g)$ from Proposition~\ref{mx:fixedtrace}. Equation~\eqref{mx:intersection-embedding} shows that the associated measure is finite:
\[
 \mu_{g,0,K}(\D)
 =\int_\D|g'|^2K(1-|z|^2)\,dA
 \lesssim\|g-g(0)\|_{Q_K^{\log}}^2<\infty.
\]
We have $f\in\ringQ$, and \eqref{mx:fixedinfinite} implies $f\ne0$. Thus $0<\|f\|_{Q_K,*}<\infty$. Normalizing the function in \eqref{eq:optimal} gives
\[
 \begin{aligned}
 C_{\ringQ}(\mu_{g,0,K})
 &\ge\int_\D\left|\frac{f}{\|f\|_{Q_K,*}}\right|^2
                    d\mu_{g,0,K}\\
 &=\frac{1}{\|f\|_{Q_K,*}^2}
       \int_\D|f|^2|g'|^2K(1-|z|^2)\,dA
 =\infty.
 \end{aligned}
\]
For each allowed $b$, apply Theorems~\ref{thm:capacity} and \ref{thm:SDP_detailed} to this finite measure. They give
\[
 C_{\ringQ}(\mu_{g,0,K})=\infty
 \quad\Longrightarrow\quad
 C^{(b)}_{K,\mathcal R}(\mu_{g,0,K})
 =D^{(b)}_{K,\mathcal R}(\mu_{g,0,K})=\infty.
\]
The same pair $(f,g)$ works as $b$ varies. Taking the suprema over $a\in\D$ includes the case $a=0$, so the final assertion follows.
\end{proof}

\endgroup

\begin{remark}\label{mx:scope}
This construction agrees with Xiao's multiplier theorem for $Q_p$ \cite[Theorem~1.2(iii)]{Xiao}. When $\beta=0$, we have $H_N=N+1$ and $\log(N+c)/\sqrt{H_N}\to0$, so the lower bound does not diverge. For $\beta=1$, it grows like $\sqrt{\log N}$, and for $\beta>1$, it grows like $\log N$. When $0<\beta<1$, we have $H_N\asymp N^{1-\beta}$, and the lower bound again tends to zero. Thus the construction leaves this range unresolved.

We also cannot replace the input functions by unnormalized monomials. The $a=0$ integral over $1/(2m)\le1-|z|^2\le1/m$ gives $$\|z^m\|_{Q_K,*}^2\gtrsim m^{1-p}/\log^\beta(em)\to\infty,$$  so their norms are not uniformly bounded. The choice $p=1/2$, $\beta=1$ alone is enough to disprove the conjecture.
\end{remark}
\medskip

\noindent\textbf{Use of Generative-AI tools declaration.} During the preparation of this work, the authors used an artificial-intelligence tool for language editing, LATEX formatting and
local mathematical reasoning. All mathematical content was written, checked, and approved by the authors, who  take full responsibility for the content of the published article.
\medskip

\noindent\textbf{Data Availability.} Data sharing is not applicable for this article as no datasets were generated or analyzed during the current study.


\bigskip

\noindent\textbf{Acknowledgement.} The authors thank Professor Hasi Wulan for inspiring lectures, insightful discussions and continuous encouragement over the years, and for bringing the $Q_K$-Carleson measure problem and the multiplier problem to the attention of the mathematical community.

\bigskip

\noindent\textbf{Conflict of interest.} The authors declare no competing interests.

\medskip

\raggedbottom

\end{document}